\documentclass[12pt]{amsart}
\usepackage{changepage}
\usepackage{multicol}
\usepackage{SSdefn}
\usepackage{SSams}
\usepackage{ulem}
\usepackage{eucal}

\makeatletter
\def\l@part{\@tocline{0}{4pt plus2pt}{0pt}{}{\bfseries}}
\makeatother

\newcommand{\init}{\mathrm{init}}

\newcommand{\surj}{\twoheadrightarrow}

\newcommand{\Conf}{\mathrm{Conf}}

\newcommand{\SYM}{\widehat{\rm Sym}}
\DeclareMathOperator{\chr}{char}
\DeclareMathOperator{\Ob}{Ob}
\DeclareMathOperator{\ch}{ch}

\newcommand{\WO}{\mathbf{WO}}
\newcommand{\OI}{\mathbf{OI}}
\newcommand{\OS}{\mathbf{OS}}
\newcommand{\FA}{\mathbf{FA}}
\newcommand{\FI}{\mathbf{FI}}
\newcommand{\FS}{\mathbf{FS}}

\newcommand{\VI}{\mathbf{VI}}
\newcommand{\VA}{\mathbf{VA}}
\newcommand{\Set}{\mathbf{Set}}

\title[Gr\"obner methods for representations of combinatorial categories]{Gr\"obner methods for representations\\ of combinatorial categories}
\date{February 8, 2016}

\author{Steven V Sam}
\address{Department of Mathematics, University of California, Berkeley, CA}
\curraddr{Department of Mathematics, University of Wisconsin, Madison, WI}
\email{\href{mailto:svs@math.wisc.edu}{svs@math.wisc.edu}}
\urladdr{\url{http://math.wisc.edu/~svs/}}

\author{Andrew Snowden}
\address{Department of Mathematics, University of Michigan, Ann Arbor, MI}
\email{\href{mailto:asnowden@umich.edu}{asnowden@umich.edu}}
\urladdr{\url{http://www-personal.umich.edu/~asnowden/}}

\thanks{SS was supported by a Miller research fellowship. AS was supported by NSF grant DMS-1303082.}

\begin{document}

\begin{abstract}
Given a category $\cC$ of a combinatorial nature, we study the following fundamental question: how do combinatorial properties of $\cC$ affect algebraic properties of representations of $\cC$? We prove two general results. The first gives a criterion for representations of $\cC$ to admit a theory of Gr\"obner bases, from which we obtain a criterion for noetherianity. The second gives a criterion for a general ``rationality'' result for Hilbert series of representations of $\cC$, and connects to the theory of formal languages.

Our work is motivated by recent work in the literature on representations of various specific categories. Our general criteria recover many of the results on these categories that had been proved by ad hoc means, and often yield cleaner proofs and stronger statements. For example: we give a new, more robust, proof that FI-modules (studied by Church, Ellenberg, and Farb), and certain generalizations, are noetherian; we prove the Lannes--Schwartz artinian conjecture from the study of generic representation theory of finite fields; we significantly improve the theory of $\Delta$-modules, introduced by Snowden in connection to syzygies of Segre embeddings; and we establish fundamental properties of twisted commutative algebras in positive characteristic.
\end{abstract}

\maketitle

\tableofcontents

\section{Introduction}

Informally, a {\bf combinatorial category} is a category $\cC$ whose objects are finite sets, possibly with extra structure, and whose morphisms are functions, possibly with extra structure. A {\bf representation} of a category over a ring $\bk$ is a functor to the category of $\bk$-modules. For simplicity of exposition, we assume that $\bk$ is a field in the introduction. Typically, a representation can be thought of as a sequence of representations of finite groups together with transition maps. Some examples of interest include:
\begin{itemize}[$\bullet$]
\item The category $\FI$ of finite sets with injections. A representation of this category can be thought of as a sequence $(M_n)_{n \ge 0}$, where $M_n$ is a representation of the symmetric group $S_n$, together with transition maps $M_n \to M_{n+1}$ satisfying certain compatibilities. This category appears in \cite{schwede} (where it is called I) in the context of stable homotopy groups, and in \cite{djament-vespa} (where it is called $\Theta$) in the context of stable homology of orthogonal and symplectic groups over finite fields. Algebraic properties of this category are studied in \cite{fimodules, fi-noeth}, where many examples of representations occurring in algebra and topology are discussed, and from a different point of view in \cite{symc1}. See \cite{farb-icm} for a survey of its uses and additional references.

\item Variants of $\FI$. In \cite{delta-mod}, modules over twisted commutative algebras are studied; these can be viewed (in certain cases) as representations of a category $\FI_d$ generalizing $\FI$. In \cite{wilson} analogs of $\FI$ for other classical Weyl groups are studied. In \cite{wiltshire} the category $\FA$ of finite sets with all maps is studied. The category $\Gamma$ of pointed finite sets (and its opposite), appears in \cite{pirashvili2} and \cite{richter} in the contexts of Hochschild homology and Goodwillie calculus.

\item The category $\FS_G$ of finite sets with $G$-surjections, $G$ being a finite group (see \S\ref{ss:G-sets} for the definition). Really, it is the opposite category that is of interest. A representation of $\FS_G^{\op}$ can be thought of as a sequence $(M_n)_{n \ge 0}$, where $M_n$ is a representation of the wreath product $S_n \wr G$, together with transition maps $M_n \to M_{n+1}$ satisfying certain compatibilities (quite different from those in the $\FI$ case). The theory of $\FS_G^{\op}$ representations, with $G$ a symmetric group, is closely tied to the theory of $\Delta$-modules studied in \cite{delta-mod}. Also, Pirashvili shows in \cite{pirashvili} that representations of $\FS^\op$ (with $G$ trivial) are equivalent to representations of $\Gamma^\op$ that send the one-point set to $0$.

\item The Brauer categories of \cite{infrank}, in which the objects are finite sets and the morphisms are certain Brauer diagrams. Representations of this category are equivalent to algebraic representations of the infinite orthogonal group.

\item The category $\VA_{\bF_q}$ of finite-dimensional vector spaces over a finite field $\bF_q$ with all linear maps. Representations of this category (in particular when $\bk = \bF_q$) have been studied in relation to algebraic K-theory, rational cohomology, and the Steenrod algebra, see \cite{kuhn:survey} for a survey and additional references. 
\end{itemize}
The referenced works use a variety of methods, often ad hoc, to study representations. However, one is struck by the fact that many of the results appear to be quite similar. For instance, each proves (or conjectures) a noetherianity result. This suggests that there are general principles at play, and leads to the subject of our paper:
\vskip .5\baselineskip
\begin{adjustwidth}{.55in}{.55in}
{\bf Main Problem.} Find practical combinatorial criteria for categories that imply interesting algebraic properties of their representations.
\end{adjustwidth}
\vskip .5\baselineskip
We give solutions to this problem for the algebraic properties of noetherianity and rationality of Hilbert series. Our criteria easily recover and strengthen most known results, and allow us to resolve some open questions. Without a doubt, they will be applicable to many categories not yet considered.

In the remainder of the introduction, we summarize our results and applications in more detail. See \S\ref{sec:intro-guide} for a guide for the paper.

\subsection{Noetherianity}

Given an object $x$ of $\cC$, define a representation $P_x$ of $\cC$ by $P_x(y)=\bk[\Hom(x,y)]$, i.e., $P_x(y)$ is the free $\bk$-module with basis $\Hom(x,y)$. We call $P_x$ the {\bf principal projective} at $x$. These representations take the place of free modules; in fact, one should think of $P_x$ as the free representation with one generator of degree $x$. A representation of $\cC$ is {\bf finitely generated} if it is a quotient of a finite direct sum of principal projective representations.

The category $\Rep_{\bk}(\cC)$ of representations of $\cC$ is called {\bf noetherian} if any subrepresentation of a finitely generated representation is again finitely generated. This is a fundamental property, and has played a crucial role in applications. The first main theoretical result of this paper is a combinatorial criterion on $\cC$ that ensures $\Rep_{\bk}(\cC)$ is noetherian, for any left-noetherian ring $\bk$. We now explain this criterion, and the motivation behind it.

We start by recalling a combinatorial proof of the Hilbert basis theorem (i.e., $\bk[x_1, \ldots, x_n]$ is a noetherian ring) using Gr\"obner bases. Pick an admissible order on the monomials, i.e., a well-order compatible with multiplication. Using the order, we can define initial ideals, and reduce the study of the ascending chain condition to monomial ideals. Now, the set of monomial ideals is naturally in bijection with the set of ideals in the poset $\bN^r$. Thus noetherianity of $\bk[x_1, \ldots, x_n]$ follows from noetherianity of the poset $\bN^r$ (Dickson's lemma), which is a simple combinatorial exercise: given infinitely many vectors $v_1, v_2, \ldots$ in $\bN^r$, there exists $i < j$ with $v_i \le v_j$, where $\le$ means coordinate-by-coordinate comparison.

To apply this method to $\Rep_{\bk}(\cC)$ we must first make sense of what ``monomials'' are. The {\bf monomials} are the basis elements $e_f \in P_x(y)$ for $f \colon x \to y$. A subrepresentation $M$ of $P_x$ is {\bf monomial} if $M(y)$ is spanned by the monomials it contains, for all $y \in \cC$.

We now carry over the combinatorial proof of the Hilbert basis theorem. For simplicity of exposition, we assume that $\cC$ is directed, i.e., if $f \colon x \to x$ then $f=\id_x$. For an object $x$ of $\cC$, we write $\vert \cC_x \vert$ for the set of isomorphism classes of morphisms $x \to y$. This can be thought of as the set of monomials in $P_x$. Suppose that $\cC$ satisfies the following condition:
\begin{itemize}
\item[(G1)] For each $x \in \cC$, the set $\vert \cC_x \vert$ admits an admissible order $\prec$, that is, a well-order compatible with post-composition, i.e., $f \prec f'$ implies $gf \prec gf'$ for all $g$.
\end{itemize}
Given a subrepresentation $M$ of a principal projective $P_x$, we can use $\prec$ to define the {\bf initial representation} $\init(M)$. This allows us to reduce the study of the ascending chain condition for subrepresentations of $P_x$ to monomial subrepresentations. Monomial subrepresentations are naturally in bijection with ideals in the poset $\vert \cC_x \vert$, where the order is defined by $f \le g$ if $g=hf$ for some $h$. (Note: $\prec$ and $\le$ are two different orders on $\vert \cC_x \vert$; the former is chosen, while the latter is canonical.) We now assume that $\cC$ satisfies one further condition:
\begin{itemize}
\item[(G2)] For each $x \in \cC$, the poset $\vert \cC_x \vert$ is noetherian.
\end{itemize}
Given this, ascending chains of monomial subrepresentations stabilize, and so $P_x$ is noetherian. One easily deduces from this that $\Rep_{\bk}(\cC)$ is a noetherian category.

The above discussion motivates one of the main definitions in this paper:

\begin{definition}
A directed category $\cC$ is {\bf Gr\"obner} if (G1) and (G2) hold.
\end{definition}

We can summarize the previous paragraph as: if $\cC$ is a directed Gr\"obner category then $\Rep_{\bk}(\cC)$ is noetherian. In the body of the paper, we give a definition of Gr\"obner that does not require directed. However, it still precludes non-trivial automorphisms, and  hence many of the categories of primary interest. This motivates a weakening of the above definition:

\begin{definition}
A category $\cC$ is {\bf quasi-Gr\"obner} if there is a Gr\"obner category $\cC'$ and an essentially surjective functor $\cC' \to \cC$ satisfying property~(F) (see Definition~\ref{def:propF}).
\end{definition}

Property~(F) is a finiteness condition that intuitively means $\cC'$ locally has finite index in $\cC$. Our main combinatorial criterion for noetherianity is the following theorem:

\begin{theorem}
If $\cC$ is quasi-Gr\"obner and $\bk$ is left-noetherian, then $\Rep_{\bk}(\cC)$ is noetherian.
\end{theorem}

This is a solution of an instance of the Main Problem: ``(quasi-)Gr\"obner'' is a purely combinatorial condition on $\cC$, which can be checked easily in practice, and the above theorem connects it to an important algebraic property of representations.

\begin{example}
Recall that $\FI$ is the category whose objects are finite sets and whose morphisms are injections. The automorphism groups in this category are symmetric groups, so $\FI$ is not  Gr\"obner. Define $\OI$ to be the category whose objects are totally ordered finite sets and whose morphisms are order-preserving injections. 
We show that $\OI$ is a Gr\"obner category (Theorem~\ref{oithm}) and the forgetful functor $\OI \to \FI$ shows that $\FI$ is quasi-Gr\"obner. In particular, we see that $\Rep_{\bk}(\FI)$ is noetherian for any left-noetherian ring $\bk$. See Remark~\ref{rmk:fi-noeth} for the history of this result and its generalizations.
\end{example}

The above example is a typical application of the theory of Gr\"obner categories: The main category of interest (in this case $\FI$) has automorphisms, and is therefore not Gr\"obner. One therefore adds extra structure (e.g., a total order) to obtain a more rigid category. One then shows that this rigidified category is Gr\"obner, which usually comes down to an explicit combinatorial problem (in this case, Dickson's lemma). Finally, one deduces that the original category is quasi-Gr\"obner, which is usually quite easy.

\subsection{Hilbert series}

A {\bf norm} on a category $\cC$ is a function $\nu \colon \vert \cC \vert \to \bN$, where $\vert \cC \vert$ is the set of isomorphism classes of $\cC$. Suppose that $\cC$ is equipped with a norm and $M$ is a representation of $\cC$ over a field $\bk$. We then define the {\bf Hilbert series} of $M$ by
\begin{displaymath}
\rH_M(t) = \sum_{x \in \vert \cC \vert} \dim_{\bk}{M(x)} \cdot t^{\nu(x)},
\end{displaymath}
when this makes sense, i.e., when the coefficient of $t^n$ is finite for all $n$. The second main theoretical result of this paper is a combinatorial condition on $(\cC, \nu)$ that ensures $\rH_M(t)$ has a particular form, for any finitely generated representation $M$.

The key idea is to connect to the theory of formal languages. Let $\Sigma$ be a finite set, and let $\Sigma^{\star}$ denote the set of all finite words in $\Sigma$ (i.e., the free monoid generated by $\Sigma$). A {\bf language} on $\Sigma$ is a subset of $\Sigma^{\star}$. Given a language $\cL$, we define its {\bf Hilbert series} by
\begin{displaymath}
\rH_{\cL}(t) = \sum_{w \in \cL} t^{\ell(w)},
\end{displaymath}
where $\ell(w)$ is the length of the word $w$. There are many known results on Hilbert series of languages: for example, regular languages have rational Hilbert series and unambiguous context-free languages have algebraic Hilbert series (Chomsky--Sch\"utzenberger theorem). We review these results, and establish some new ones, in \S \ref{s:lang}.

Let $\cP$ be a class of languages. A {\bf $\cP$-lingual structure} on $\cC$ at $x$ consists of a finite alphabet $\Sigma$ and an injection $i \colon \vert \cC_x \vert \to \Sigma^{\star}$ (i.e., a way of regarding monomials in $P_x$ as words) such that the following two conditions hold. First, for any $f \colon x \to y$ in $\vert \cC_x \vert$ we have $\nu(y)=\ell(i(f))$, that is, the norm of a monomial coincides with the length of the corresponding word. And second, if $S$ is a poset ideal of $\vert \cC_x \vert$ then the language $i(S)$ is of class $\cP$. We say that $\cC$ is {\bf $\cP$-lingual} if it admits a $\cP$-lingual structure at every object. A special case of our main result on Hilbert series is then:

\begin{theorem}
Suppose that $\cC$ is a $\cP$-lingual Gr\"obner category and let $M$ be a finitely generated representation of $\cC$. Then $\rH_M$ is a $\bZ$-linear combination of series of the form $\rH_{\cL}$, where each $\cL$ is a language of class $\cP$.
\end{theorem}

This too is a solution of an instance of the Main Problem: ``$\cP$-lingual'' is a purely combinatorial condition on $(\cC, \nu)$, which can be easily checked in practice, and the above theorem connects it to an important algebraic property of representations.

\begin{example}
Define a norm on $\OI$ by $\nu(x)=\# x$. We show that $(\OI, \nu)$ is $\cP$-lingual, where $\cP$ is the class of ``ordered languages'' introduced in \S \ref{ss:ordered}; we refer to Theorem~\ref{oithm} for the details. As a consequence, if $M$ is a finitely generated $\FI$-module then $\rH_M(t) = \frac{p(t)}{(1-t)^k}$ where $p(t)$ is a polynomial and $k \ge 0$, and so the function $n \mapsto \dim_{\bk} M([n])$ is eventually polynomial. See Remark~\ref{rmk:fi-history} for the history of this result and its generalizations.
\end{example}

\subsection{Applications}

We apply our theory to prove a large number of results about categories of interest. We mention three of these results here.

\subsubsection{Lannes--Schwartz artinian conjecture} \label{sec:intro-schwartz}

This conjecture, posed by Jean Lannes and Lionel Schwartz in the late 1980s, is equivalent to the assertion that $\Rep_{\bF_q}(\VA_{\bF_q})$ is noetherian, i.e., the principal projectives $P_V$ are noetherian for all $V \in \VA_{\bF_q}$. As originally stated, the conjecture
asserts the equivalent dual statement that the principal injectives are artinian. This first appeared in print as \cite[Conjecture B.12]{kuhnI}, and again in \cite[Conjecture 3.12]{kuhnII}. See \cite[Proposition 3.13]{kuhnII} for a number of equivalent formulations; one such formulation is that any finitely generated functor has a resolution by finitely generated projectives, so that Ext modules between finitely generated functors are finite dimensional. 

Previous work on this had been of two sorts. Papers \cite{DjamentFoncteurs, DjamentLeFoncteur, Djament-finite, KuhnInvariant, PowellArtinian, PowellStructure, PowellArtinianObjects} focus on the structure of $P_V$, and, in particular, show that $P_V$ is noetherian for $\dim V = 1$ and all $q$, and $\dim V \le 3$ and $q = 2$. In a different direction, Schwartz showed that any polynomial functor taking finite dimensional values has a resolution by finitely generated projectives. See \cite[Theorem 5.3.8]{schwartz} for the original proof, and \cite[Proposition 10.1]{FLS} for a slightly more general result of this sort.

The conjecture is a special case of Corollary~\ref{VI-noeth}. A similar (but distinct) proof of this conjecture appears in \cite{putman-sam}. Streamlined proofs of this conjecture based on these works can be found in \cite{djament-bourbaki, krause}.

\subsubsection{Syzygies of Segre embeddings}
In \cite{delta-mod}, the $p$-syzygies of all Segre embeddings (with any number of factors of any dimensions, but with $p$ fixed) are assembled into a single algebraic structure called a $\Delta$-module. The two main results of \cite{delta-mod} state that, in characteristic $0$, this structure is finitely generated and has a rational Hilbert series. Informally, these results mean that the $p$-syzygies of Segre embeddings admit a finite description.

We improve the results of \cite{delta-mod} in three ways. First, we show that the main theorems continue to hold in positive characteristic. We remark that the syzygies of the Segre embeddings are known to behave differently in positive characteristic (see \cite{hashimoto} for determinantal ideals, which for $2 \times 2$ determinants are special cases of Segre embeddings). While we only work with $\Delta$-modules over fields, one can work over $\bZ$ and show that for any given syzygy module, the type of torsion that appears is bounded.

Second, we greatly improve the rationality result for the Hilbert series and give an affirmative answer to \cite[Question~5]{delta-mod}. This yields strong constraints on how $p$-syzygies of the Segre embeddings vary with the number of factors.

Finally, we remove a technical assumption from \cite{delta-mod}: that paper only dealt with ``small'' $\Delta$-modules, whereas our methods handle all finitely generated $\Delta$-modules. This is useful for technical reasons: for instance, it shows that a finitely generated $\Delta$-module admits a resolution by finitely generated projective $\Delta$-modules. In particular, we find that $\Delta$-modules of syzygies are finitely presented, and so can be completely described in a finite manner.

It is worth emphasizing that this paper greatly simplifies the approach to $\Delta$-modules. The treatment in \cite{delta-mod} is abstract and complicated. Ours is much more elementary and direct.

\subsubsection{Twisted commutative algebras in positive characteristic}
A twisted commutative algebra (tca) is a graded algebra on which the symmetric group $S_n$ acts on the $n$th graded piece so that the multiplication is commutative up to a ``twist'' by the symmetric group. In characteristic $0$, one can use Schur--Weyl duality to describe tca's in terms of commutative algebras with an action of $\GL(\infty)$, and we have fruitfully exploited this to obtain many results \cite{delta-mod,symc1,expos,infrank}. This method is inapplicable in positive characteristic, so we know much less about tca's there. In \cite{fi-noeth}, the univariate tca $\bk \langle x \rangle$ was analyzed, for any ring $\bk$, and certain fundamental results (such as noetherianity) were established. We establish many of the same results for the multivariate tca $\bk\langle x_1, \ldots, x_d \rangle$; see \S \ref{ss:tca} for details.

\subsection{Relation to previous work}

\begin{itemize}[$\bullet$]
\item The idea of reducing noetherianity of an algebraic structure to that of a poset has been used before in different contexts. We highlight \cite{cohen} for an example in universal algebra and \cite{higman} for examples in abstract algebra. After writing a first draft of this article, Henning Krause notified us that the statement of Theorem~\ref{grobnoeth} for Gr\"obner categories appears in \cite[Theorem 3.1]{richter2}.

\item The notion of quasi-Gr\"obner formalizes the intuitive idea of ``breaking symmetry.'' A similar idea is used in \cite{DK-operad} to define Gr\"obner bases for symmetric operads by passing to the weaker structure of shuffle operads. This idea is used in \cite{KP-operad} to study Hilbert series of operads with well-behaved Gr\"obner bases.

\item A related topic (and one that serves as motivation for us) is the notion of ``noetherianity up to symmetry'' in multilinear algebra and algebraic statistics. We point to \cite{draismaeggermont, draismakuttler, hillarmartin, hillarsullivant} for some applications of noetherianity and to \cite{draisma-notes} for a survey and further references.

\item While preparing this article, we discovered that an essentially equivalent version of  Proposition~\ref{prop:secondorder-noeth} is proven in the proof of \cite[Proposition 7.5]{draismakuttler}. 

\item Another source of motivation is the topic of representation stability \cite{church-farb, farb-icm} and its relation to $\FI$-modules \cite{fimodules, fi-noeth}. The main advantage of our approach over the ones in \cite{fimodules, fi-noeth} is its flexibility. For example, it is remarked after \cite[Proposition 2.12]{fi-noeth} that their techniques cannot handle linear analogues of $\FI$; we handle these in Theorem~\ref{thm:VI-noeth}.
\end{itemize}

\subsection{Guide to the paper} \label{sec:intro-guide}

The paper is divided into two parts: theory and applications. 

The first part on theory begins with background results on noetherian posets (\S\ref{sec:posets}). The material on posets is standard and we have included proofs to make it self-contained; we use it for applications to noetherian conditions on representations of categories. 

In \S\ref{sec:rep-cat} we introduce basic terminology and properties of representations of categories and functors between them. We state criteria for noetherian conditions on representations which will be further developed in later sections. \S\ref{sec:noeth-grob} introduces and develops the main topic of this paper: Gr\"obner bases for representations of categories and (quasi-)Gr\"obner categories. We give a criterion for categories to admit a Gr\"obner basis theory and relate noetherian properties of representations to those of posets. 

\S\ref{s:lang} discusses formal languages and is a mix of review and new material on ordered languages. This material is essential for \S\ref{sec:ling}, which contains our applications to Hilbert series of representations of categories. Here we introduce the notion of lingual structures on categories and connect properties of Hilbert series with formal languages. 

The second part of the paper contains applications of the theory developed in the first part. \S\ref{sec:cat-inj} is about categories of finite sets and injective functions of different kinds. This has two sources of motivation: the theory of $\FI$-modules \cite{fimodules} and the theory of twisted commutative algebras \cite{expos}. We recover and strengthen known results on noetherianity and Hilbert series for these categories and related ones. 

\S\ref{sec:cat-surj} concerns categories of finite sets and surjective functions. These categories are much more complicated than their injective counterparts. A significant application of the results here is the proof of the Lannes--Schwartz artinian conjecture (discussed in \S\ref{sec:intro-schwartz}). 

\S\ref{sec:Delta-mod} is about applications to $\Delta$-modules (introduced by the second author in \cite{delta-mod}). We prove that finitely generated $\Delta$-modules are noetherian over any field, thus significantly improving the results of \cite{delta-mod} where it is only shown in characteristic $0$ under a ``smallness'' assumption. For Hilbert series of $\Delta$-modules, we affirmatively resolve and strengthen \cite[\S 6, Question 5]{delta-mod} by proving a stronger rationality result.

In \S\ref{sec:additional}, we mention some additional examples, some of which will appear in a different article for reasons of space. Finally, we end with some open problems in \S\ref{sec:problems}.

\subsection{Notation.}

\begin{itemize}
\item If $\Sigma$ is a set, let $\Sigma^\star$ denote the set of words in $\Sigma$, i.e., the free monoid generated by $\Sigma$. For $w \in \Sigma^{\star}$, let $\ell(w)$ denote the length of the word.
\item For a non-negative integer $n \ge 0$, set $[n] = \{1, \dots, n\}$, with the convention $[0] = \emptyset$.
\item Let $\bn$ be an element of $\bN^r$. We write $\vert \bn \vert$ for the sum of the coordinate of $\bn$ and $\bn!$ for $n_1! \cdots n_r!$. We let $[\bn]$ be the tuple $([n_1], \ldots, [n_r])$ of finite sets. Given variables $t_1, \ldots, t_r$, we let $\bt^{\bn}$ be the monomial $t_1^{n_1} \cdots t_r^{n_r}$.
\item $\SYM$ denotes the completion of $\Sym$ with respect to the homogeneous maximal ideal, i.e., $\SYM(V) = \prod_{n \ge 0} \Sym^n(V)$.
\end{itemize}

{\bf Acknowledgements.}
We thank Aur\'elien Djament, Benson Farb, Nicholas Kuhn, Andrew Putman, and Bernd Sturmfels for helpful discussions and comments. We also thank some anonymous referees for their careful reading of the paper and their invaluable suggestions which greatly improved the paper.

\part{Theory}

\section{Partially ordered sets} \label{sec:posets}

In this section, we state some basic definitions and properties of noetherian posets. This section can be skipped and referred back to as necessary since it serves a technical role only.

Let $X$ be a poset. Then $X$ satisfies the {\bf ascending chain condition} (ACC) if every ascending chain in $X$ stabilizes, i.e., given $x_1 \le x_2 \le \cdots$ in $X$ we have $x_i=x_{i+1}$ for $i \gg 0$. The {\bf descending chain condition} (DCC) is defined similarly. An {\bf anti-chain} in $X$ is a sequence $x_1, x_2, \ldots$ such that $x_i \nlet x_j$ for all $i \ne j$. An {\bf ideal} in $X$ is a subset $I$ of $X$ such that $x \in I$ and $x \le y$ implies $y \in I$. We write $\cI(X)$ for the poset of ideals of $X$, ordered by inclusion. For $x \in X$, the {\bf principal ideal} generated by $x$ is $\{y \mid y \ge x\}$. An ideal is {\bf finitely generated} if it is a finite union of principal ideals. The following result is standard.

\begin{proposition}
The following conditions on $X$ are equivalent:
\begin{enumerate}[\indent \rm (a)]
\item The poset $X$ satisfies DCC and has no infinite anti-chains.
\item Given a sequence $x_1, x_2, \ldots$ in $X$, there exists $i<j$ such that $x_i \le x_j$.
\item The poset $\cI(X)$ satisfies ACC.
\item Every ideal of $X$ is finitely generated.
\end{enumerate}
\end{proposition}

The poset $X$ is {\bf noetherian} if the above conditions are satisfied. Where we say ``$X$ is noetherian,'' one often sees ``$\le$ is a well-quasi-order'' in the literature. Similarly, where we say ``$X$ satisfies DCC'' one sees ``$\le$ is well-founded.''

\begin{proposition} \label{wqo-inf}
Let $X$ be a noetherian poset and let $x_1, x_2, \dots$ be a sequence in $X$. Then there exists an infinite sequence of indices $i_1 < i_2 < \cdots$ such that $x_{i_1} \le x_{i_2} \le \cdots$.
\end{proposition}

\begin{proof}
Let $I$ be the set of indices such that $i \in I$ and $j > i$ implies that $x_i \nlet x_j$. If $I$ is infinite, then there is $i < i'$ with $i,i' \in I$ such that $x_i \le x_{i'}$ by definition of noetherian and hence contradicts the definition of $I$. So $I$ is finite; let $i_1$ be any number larger than all elements of $I$. Then by definition of $I$, we can find $x_{i_1} \le x_{i_2} \le \cdots$.
\end{proof}

\begin{proposition} \label{noethprod}
Let $X$ and $Y$ be noetherian posets. Then $X \times Y$ is noetherian.
\end{proposition}

\begin{proof}
Let $(x_1, y_1), (x_2, y_2), \dots$ be an infinite sequence in $X \times Y$. Since $X$ is noetherian, there exists $i_1 < i_2 < \cdots$ such that $x_{i_1} \le x_{i_2} \le \cdots$ (Proposition~\ref{wqo-inf}). Since $Y$ is noetherian, there exists $i_j < i_{j'}$ such that $y_{i_j} \le y_{i_{j'}}$, and hence $(x_{i_j}, y_{i_j}) \le (x_{i_{j'}}, y_{i_{j'}})$.
\end{proof}

Let $X$ and $Y$ be posets and let $f \colon X \to Y$ be a function. We say that $f$ is {\bf order-preserving} if $x \le x'$ implies $f(x) \le f(x')$, and $f$ is {\bf strictly order-preserving} if $x \le x'$ is equivalent to $f(x) \le f(x')$. If $f$ is strictly order-preserving, then it is injective; we can thus regard $X$ as a subposet of $Y$. In particular, if $Y$ is noetherian then so is $X$.

Let $\cF=\cF(X,Y)$ be the set of all order-preserving functions $f \colon X \to Y$. We partially order $\cF$ by $f \le g$ if $f(x) \le g(x)$ for all $x \in X$.

\begin{proposition} \label{prop:poset-fns}
We have the following.
\begin{enumerate}[\indent \rm (a)]
\item If $X$ is noetherian and $Y$ satisfies ACC then $\cF$ satisfies ACC. 
\item If $\cF$ satisfies ACC and $X$ is non-empty then $Y$ satisfies ACC.
\item If $\cF$ satisfies ACC and $Y$ has two distinct comparable elements then $X$ is noetherian.
\end{enumerate}
\end{proposition}

\begin{proof}
(a) Suppose $X$ is noetherian and $\cF$ does not satisfy ACC. Let $f_1<f_2<\cdots$ be an ascending chain in $\cF$. For each $i$, choose $x_i \in X$ such that $f_i(x_i)<f_{i+1}(x_i)$. By passing to a subsequence we can assume $x_1 \le x_2 \le \cdots$ (Proposition~\ref{wqo-inf}). Let $y_i=f_i(x_i)$. Then $f_i(x_i)<f_{i+1}(x_i) \le f_{i+1}(x_{i+1})$ implies that $y_1<y_2<\cdots$, so $Y$ does not satisfy ACC.

(b) Now suppose $\cF$ satisfies ACC and $X$ is non-empty. Then $Y$ embeds into $\cF$ as the set of constant functions, and so $Y$ satisfies ACC. 

(c) Finally, suppose $\cF$ satisfies ACC and $Y$ contains elements $y_1<y_2$. Given an ideal $I$ of $X$, define $\chi_I \in \cF$ by
\begin{displaymath}
\chi_I(x) = \begin{cases} y_2 & x \in I \\ y_1 & x \not\in I \end{cases}.
\end{displaymath}
Then $I \mapsto \chi_I$ embeds $\cI(X)$ into $\cF$, and so $\cI(X)$ satisfies ACC, and so $X$ is noetherian.
\end{proof}

Given a poset $X$, let $X^\star$ be the set of finite words $x_1 \cdots x_n$ with $x_i \in X$. We define $x_1 \cdots x_n \le x'_1 \cdots x'_m$ if there exist $1 \le i_1 < \cdots < i_n \le m$ such that $x_j \le x'_{i_j}$ for $j=1,\dots,n$.

\begin{theorem}[Higman's lemma \cite{higman}] \label{thm:higman}
If $X$ is a noetherian poset, then so is $X^\star$.
\end{theorem}

\begin{proof}
Suppose that $X^\star$ is not noetherian. We use Nash-Williams' theory of minimal bad sequences \cite{nashwilliams} to get a contradiction. A sequence $w_1,w_2,\ldots$ of elements in $X^\star$ is bad if $w_i \nlet w_j$ for all $i<j$. We pick a bad sequence minimal in the following sense: for all $i \ge 1$, among all bad sequences beginning with $w_1, \dots, w_{i-1}$ (this is the empty sequence for $i=1$), $\ell(w_i)$ is as small as possible. Let $x_i \in X$ be the first element of $w_i$ and let $v_i$ be the subword of $w_i$ obtained by removing $x_i$. By Proposition~\ref{wqo-inf}, there is an infinite sequence $i_1 < i_2 < \cdots$ such that $x_{i_1} \le x_{i_2} \le \cdots$. Then $w_1, w_2, \dots, w_{i_1-1}, v_{i_1}, v_{i_2}, \dots$ is a bad sequence because $v_{i_j} \le w_{i_j}$ for all $j$, and $v_{i_j} \le v_{i_{j'}}$ would imply that $w_{i_j} \le w_{i_{j'}}$. It is smaller than our minimal bad sequence, so we have reached a contradiction. Hence $X^\star$ is noetherian.
\end{proof}

\section{Representations of categories}
\label{sec:rep-cat}

This section introduces the main topic of this paper: representations of categories. Our goal is to lay out the main definitions and basic properties of representations and functors between categories of representations and to state some criteria for representations to be noetherian. More specifically, definitions are given in \S\ref{ss:repdefn}, properties of functors between categories are in \S\ref{sec:pullbackfunctor}, and tensor products of representations are discussed in \S\ref{sec:tensor-prod}.

\subsection{Basic definitions and results} \label{ss:repdefn}

Let $\cC$ be an essentially small category. We denote by $\vert \cC \vert$ the set of isomorphism classes in $\cC$. For an object $x$ of $\cC$, we let $\cC_x$ be the category of morphisms from $x$; thus the objects of $\cC_x$ are morphisms $x \to y$ (with $y$ variable), and the morphisms in $\cC_x$ are the obvious commutative triangles. We say that $\cC$ is {\bf directed} if every self-map in $\cC$ is the identity. If $\cC$ is directed then so is $\cC_x$, for any $x$. If $\cC$ is essentially small and directed, then $\vert \cC \vert$ is naturally a poset by defining $x \le y$ if there exists a morphism $x \to y$. We say that $\cC$ is $\Hom$-finite if all $\Hom$ sets are finite.

Fix a nonzero ring $\bk$ (not necessarily commutative) and let $\Mod_{\bk}$ denote the category of left $\bk$-modules. A {\bf representation} of $\cC$ (or a {\bf $\cC$-module}) over $\bk$ is a functor $\cC \to \Mod_{\bk}$. A map of $\cC$-modules is a natural transformation. We write $\Rep_{\bk}(\cC)$ for the category of representations, which is abelian. Let $M$ be a representation of $\cC$. By an {\bf element} of $M$ we mean an element of $M(x)$ for some object $x$ of $\cC$. Given any set $S$ of elements of $M$, there is a smallest subrepresentation of $M$ containing $S$; we call this the subrepresentation {\bf generated by} $S$. We say that $M$ is {\bf finitely generated} if it is generated by a finite set of elements. For a morphism $f \colon x \to y$ in $\cC$, we typically write $f_*$ for the given map of $\bk$-modules $M(x) \to M(y)$.

Let $x$ be an object of $\cC$. Define a representation $P_x$ of $\cC$ by $P_x(y)=\bk[\Hom(x,y)]$, i.e., $P_x(y)$ is the free left $\bk$-module with basis $\Hom(x,y)$. For a morphism $f \colon x \to y$, we write $e_f$ for the corresponding element of $P_x(y)$. If $M$ is another representation then $\Hom(P_x, M)=M(x)$. This shows that $\Hom(P_x,-)$ is an exact functor, and so $P_x$ is a projective object of $\Rep_{\bk}(\cC)$. We call it the {\bf principal projective} at $x$. A $\cC$-module is finitely generated if and only if it is a quotient of a finite direct sum of principal projective objects. 

An object of $\Rep_{\bk}(\cC)$ is {\bf noetherian} if every ascending chain of subobjects stabilizes; this is equivalent to every subrepresentation being finitely generated. The category $\Rep_{\bk}(\cC)$ is {\bf noetherian} if every finitely generated object in it is.

\begin{proposition} \label{cat-noeth}
The category $\Rep_{\bk}(\cC)$ is noetherian if and only if every principal projective is noetherian.
\end{proposition}

\begin{proof}
Obviously, if $\Rep_{\bk}(\cC)$ is noetherian then so is every principal projective. Conversely, suppose every principal projective is noetherian. Let $M$ be a finitely generated object. Then $M$ is a quotient of a finite direct sum $P$ of principal projectives. Since noetherianity is preserved under finite direct sums, $P$ is noetherian. And since noetherianity descends to quotients, $M$ is noetherian. This completes the proof.
\end{proof}

\subsection{Pullback functors} \label{sec:pullbackfunctor}

Let $\Phi \colon \cC \to \cC'$ be a functor. There is then a pullback functor on representations $\Phi^* \colon \Rep_{\bk}(\cC') \to \Rep_{\bk}(\cC)$. In this section, we study how $\Phi^*$ interacts with finiteness conditions. The following definition is of central importance:

\begin{definition} \label{def:propF}
We say that $\Phi$ satisfies {\bf property~(F)} (F for finite) if the following condition holds: given any object $x$ of $\cC'$ there exist finitely many objects $y_1, \ldots, y_n$ of $\cC$ and morphisms $f_i \colon x \to \Phi(y_i)$ in $\cC'$ such that for any object $y$ of $\cC$ and any morphism $f \colon x \to \Phi(y)$ in $\cC'$, there exists a morphism $g \colon y_i \to y$ in $\cC$ such that $f=\Phi(g) \circ f_i$.
\end{definition}

\begin{remark} \label{Fadjoint}
If $\Phi$ has a left adjoint $\Psi$ then it automatically satisfies property~(F): indeed, one can take $n=1$, $y_1=\Psi(x)$, and $f_1$ the unit $x \to \Phi(\Psi(x))$. We thank a referee for this remark.
\end{remark}

The following proposition is the motivation for introducing property~(F).

\begin{proposition} \label{propFfg}
A functor $\Phi \colon \cC \to \cC'$ satisfies property {\rm (F)} if and only if $\Phi^*$ takes finitely generated objects of $\Rep_{\bk}(\cC')$ to finitely generated objects of $\Rep_{\bk}(\cC)$.
\end{proposition}

\begin{proof}
Assume that $\Phi$ satisfies property (F). It suffices to show that $\Phi^*$ takes principal projectives to finitely generated representations. Thus let $P_x$ be the principal projective of $\Rep_{\bk}(\cC')$ at an object $x$. Note that $\Phi^*(P_x)(y)$ has for a basis the elements $e_f$ for $f \in \Hom_{\cC'}(x, \Phi(y))$. Let $f_i \colon x \to \Phi(y_i)$ be as in the definition of property (F). Then the $e_{f_i}$ generate $\Phi^*(P_x)$. The converse is left to the reader (and not used in this paper).
\end{proof}

\begin{proposition} \label{pullback-fg}
Suppose that $\Phi \colon \cC \to \cC'$ is an essentially surjective functor. Let $M$ be an object of $\Rep_{\bk}(\cC')$ such that $\Phi^*(M)$ is finitely generated {\rm (}resp.\ noetherian{\rm )}. Then $M$ is finitely generated {\rm (}resp.\ noetherian{\rm )}.
\end{proposition}

\begin{proof}
Let $S$ be a set of elements of $\Phi^*(M)$. Let $S'$ be the corresponding set of elements of $M$. (Thus if $S$ contains $m \in \Phi^*(M)(y)$ then $S'$ contains $m \in M(\Phi(y))$.) If $N$ is a subrepresentation of $M$ containing $S'$ then $\Phi^*(N)$ is a subrepresentation of $\Phi^*(M)$ containing $S$. It follows that if $N$ (resp.\ $N'$) is the subrepresentation of $M$ (resp.\ $\Phi^*(M)$) generated by $S'$ (resp.\ $S$), then $N' \subset \Phi^*(N)$. Thus if $S$ generates $\Phi^*(M)$ then $\Phi^*(N)=\Phi^*(M)$, which implies $N=M$ since $\Phi$ is essentially surjective, i.e., $S$ generates $M$. In particular, if $\Phi^*(M)$ is finitely generated then so is $M$.

Now suppose that $\Phi^*(M)$ is noetherian. Given a subrepresentation $N$ of $M$, we obtain a subrepresentation $\Phi^*(N)$ of $\Phi^*(M)$. Since $\Phi^*(M)$ is noetherian, it follows that $\Phi^*(N)$ is finitely generated. Thus $N$ is finitely generated, and so $M$ is noetherian.
\end{proof}

\begin{corollary} \label{pullback}
Let $\Phi \colon \cC \to \cC'$ be an essentially surjective functor satisfying property~{\rm (F)} and suppose $\Rep_{\bk}(\cC)$ is noetherian. Then $\Rep_{\bk}(\cC')$ is noetherian.
\end{corollary}

\begin{proof}
Let $M$ be a finitely generated $\cC'$-module. Then $\Phi^*(M)$ is finitely generated by Proposition~\ref{propFfg}, and therefore noetherian, and so $M$ is noetherian by Proposition~\ref{pullback-fg}.
\end{proof}

The next two results follow from the definitions, so we omit their proofs.

\begin{proposition} \label{propFcomp}
Suppose $\Phi \colon \cC_1 \to \cC_2$ and $\Psi \colon \cC_2 \to \cC_3$ satisfy property~{\rm (F)}. Then the composition $\Psi \circ \Phi$ satisfies property~{\rm (F)}.
\end{proposition}

\begin{proposition} \label{propGcancel}
Let $\Phi \colon \cC_1 \to \cC_2$ and $\Psi \colon \cC_2 \to \cC_3$ be functors. Suppose that $\Phi$ is essentially surjective and $\Psi \circ \Phi$ satisfies property~{\rm (F)}. Then $\Psi$ satisfies property~{\rm (F)}.
\end{proposition}

\subsection{Tensor products} \label{sec:tensor-prod}

Suppose $\bk$ is a commutative ring. Given representations $M$ and $N$ of categories $\cC$ and $\cD$ over $\bk$, we define their {\bf external tensor product}, denoted $M \boxtimes N$, to be the representation of $\cC \times \cD$ given by $(x, y) \mapsto M(x) \otimes_{\bk} N(y)$. One easily sees that if $M$ and $N$ are finitely generated then so is $M \boxtimes N$. If $\cC=\cD$ then we define the {\bf pointwise tensor product} of $M$ and $N$, denoted $M \odot N$, to be the representation of $\cC$ given by $x \mapsto M(x) \otimes_{\bk} N(x)$. The two tensor products are related by the identity $M \odot N = \Delta^*(M \boxtimes N)$, where $\Delta \colon \cC \to \cC \times \cC$ is the diagonal functor. From this discussion, we have the following definition and result.

\begin{definition} \label{def:catF} 
A category $\cC$ satisfies {\bf property~(F)} if the diagonal functor $\Delta \colon \cC \to \cC \times \cC$ satisfies property~(F). 
\end{definition}

\begin{proposition} \label{pointwise}
If $\cC$ satisfies property~{\rm (F)} then the pointwise tensor product of finitely generated representations is finitely generated.
\end{proposition}

\begin{remark} 
If $\cC$ has coproducts then it automatically satisfies property~(F), as the coproduct provides a left adjoint to the diagonal (see Remark~\ref{Fadjoint}).
\end{remark}

\section{Noetherianity and Gr\"obner categories}
\label{sec:noeth-grob}

This section introduces another main topic of this paper: Gr\"obner bases for representations of categories. We first discuss monomial representations in \S\ref{ss:monomial}. Definitions and basic properties of Gr\"obner bases and initial representations are given in \S\ref{ss:grobner} and we state a Gr\"obner-theoretic approach to proving the noetherian property. In \S\ref{ss:grob-cat}, we introduce the notions of Gr\"obner and quasi-Gr\"obner categories, which are those categories for which the formalism of the first section can be applied. 

\subsection{Monomial representations} \label{ss:monomial}

Recall that $\bk$ is a general ring. Let $\cC$ be an essentially small category and let $\Set$ be the category of sets. Fix a functor $S \colon \cC \to \Set$, and let $P=\bk[S]$, i.e., $P(x)$ is the free $\bk$-module on the set $S(x)$. 

We begin by defining a partially ordered set $\vert S \vert$ associated to $S$, which is one of the main combinatorial objects of interest. A subfunctor of $S$ is {\bf principal} if it is generated by a single element. (Here we use ``element'' and ``generated'' as with representations of $\cC$.) We define $\vert S \vert$ to be the set of principal subfunctors of $S$, partially ordered by reverse inclusion. We can describe this poset more concretely as follows, at least when $\cC$ is small. Set $\wt{S}=\bigcup_{x \in \cC} S(x)$. Given $f \in S(x)$ and $g \in S(y)$, define $f \le g$ if there exists $h \colon x \to y$ with $h_*(f)=g$. Define an equivalence relation on $\wt{S}$ by $f \sim g$ if $f \le g$ and $g \le f$. The poset $\vert S \vert$ is the quotient of $\wt{S}$ by $\sim$, with the induced partial order.

Given $f \in S(x)$, we write $e_f$ for the corresponding element of $P(x)$. An element of $P$ is a {\bf monomial} if it is of the form $\lambda e_f$ for some $\lambda \in \bk$ and $f \in S(x)$. A subrepresentation $M$ of $P$ is {\bf monomial} if $M(x)$ is spanned by the monomials it contains, for all objects $x$.

We now classify the monomial subrepresentations of $P$ in terms of $\vert S \vert$. Given $f \in \wt{S}$, let $I_M(f) = \{\lambda \in \bk \mid \lambda e_f \in M\}$. Then $I_M(f)$ is an ideal of $\bk$. If $f \le g$ then $I_M(f) \subset I_M(g)$. In particular, $I_M(f)=I_M(g)$ if $f \sim g$. Let $\cI(\bk)$ be the poset of left-ideals in $\bk$ and let $\cM(P)$ be the poset of monomial subrepresentations of $P$ (ordered by inclusion). Given $M \in \cM(P)$, we have constructed an order-preserving function $I_M \colon \vert S \vert \to \cI(\bk)$, i.e., an element of $\cF(\vert S \vert, \cI(\bk))$ (see \S\ref{sec:posets}).

\begin{proposition} \label{monclass}
The map $I \colon \cM(P) \to \cF(\vert S \vert, \cI(\bk))$ is an isomorphism of posets.
\end{proposition}

\begin{proof}
Suppose that for every $f \in \vert S \vert$ we have a left-ideal $I(f)$ of $\bk$ such that for $f \le g$ we have $I(f) \subseteq I(g)$. We then define a monomial subrepresentation $M \subseteq P$ by $M(x)=\sum_{f \in S(x)} I(f) e_f$. This defines a function $\cF(\vert S \vert, \cI(\bk)) \to \cM(P)$ inverse to $I$. It is clear from the constructions that $I$ and its inverse are order-preserving, and so $I$ is an isomorphism of posets.
\end{proof}

\begin{corollary} \label{lem:mon-noeth}
The following are equivalent {\rm (}assuming $P$ is non-zero{\rm )}.
\begin{enumerate}[\indent \rm (a)]
\item Every monomial subrepresentation of $P$ is finitely generated.
\item The poset $\cM(P)$ satisfies ACC.
\item The poset $\vert S \vert$ is noetherian and $\bk$ is left-noetherian.
\end{enumerate}
\end{corollary}

\begin{proof}
The equivalence of (a) and (b) is standard, while the equivalence of (b) and (c) follows from Propositions~\ref{monclass} and~\ref{prop:poset-fns}. (Note: $\vert S \vert$ is non-empty if $P \ne 0$, and $\cI(\bk)$ contains two distinct comparable elements, namely the zero and unit ideals.)
\end{proof}

\subsection{Gr\"obner bases} \label{ss:grobner}

Continue the notation of the previous section. The purpose of this section is to develop a theory of Gr\"obner bases for $P$, and use this theory to give a combinatorial criterion for $P$ to be noetherian.

To connect arbitrary subrepresentations of $P$ to monomial subrepresentations, we need a theory of monomial orders. Let $\WO$ be the category of well-ordered sets and strictly order-preserving functions. There is a forgetful functor $\WO \to \Set$. An {\bf ordering} on $S$ is a lifting of $S$ to $\WO$. More concretely, an ordering on $S$ is a choice of well-order on $S(x)$, for each $x \in \cC$, such that for every morphism $x \to y$ in $\cC$ the induced map $S(x) \to S(y)$ is strictly order-preserving. We write $\preceq$ for an ordering; $S$ is {\bf orderable} if it admits an ordering.

Suppose $\preceq$ is an ordering on $S$. Given non-zero $\alpha = \sum_{f \in S(x)} \lambda_f e_f$ in $P(x)$, we define the {\bf initial term} of $\alpha$, denoted $\init(\alpha)$, to be $\lambda_g e_g$, where $g=\max_{\preceq} \{ f \mid \lambda_f \ne 0 \}$. The {\bf initial variable} of $\alpha$, denoted $\init_0(\alpha)$, is $g$. Now let $M$ be a subrepresentation of $P$. We define the {\bf initial representation} of $M$, denoted $\init(M)$, as follows: $\init(M)(x)$ is the $\bk$-span of the elements $\init(\alpha)$ for non-zero $\alpha \in M(x)$. The name is justified by the following result.

\begin{proposition}
Notation as above, $\init(M)$ is a monomial subrepresentation of $P$.
\end{proposition}

\begin{proof}
Let $\alpha = \sum_{i=1}^n \lambda_i e_{f_i}$ be an element of $M(x)$ with each $\lambda_i$ non-zero, ordered so that $f_i \prec f_1$ for all $i>1$. Thus $\init(\alpha)=\lambda_1 e_{f_1}$. Let $g \colon x \to y$ be a morphism. Then $g_*(\alpha)=\sum_{i=1}^n \lambda_i e_{g_*(f_i)}$. Since $g_* \colon S(x) \to S(y)$ is strictly order-preserving, we have $g_*(f_i) \prec g_*(f_1)$ for all $i>1$. Thus $\init(g_*(\alpha))=\lambda_1 e_{gf_1}$, or, in other words, $\init(g_*(\alpha))=g_*(\init(\alpha))$. This shows that $g_*$ maps $\init(M)(x)$ into $\init(M)(y)$, and so $\init(M)$ is a subrepresentation of $P$. That it is monomial follows immediately from its definition.
\end{proof}

\begin{proposition} \label{initeq}
If $N \subseteq M$ are subrepresentations of $P$ and $\init(N)=\init(M)$, then $M=N$.
\end{proposition}

\begin{proof}
Assume that $M(x) \ne N(x)$ for $x \in \cC$. Let $K \subset S(x)$ be the set of all elements which appear as the initial variable of some element of $M(x) \setminus N(x)$. Then $K \ne \emptyset$, so has a minimal element $f$ with respect to $\preceq$. Pick $\alpha \in M(x) \setminus N(x)$ with $\init_0(\alpha)=f$. By assumption, there exists $\beta \in N(x)$ with $\init(\alpha)=\init(\beta)$. But then $\alpha-\beta \in M(x) \setminus N(x)$, and $\init_0(\alpha-\beta) \prec \init_0(\alpha)$, a contradiction. Thus $M=N$.
\end{proof}

Let $M$ be a subrepresentation of $P$. A set $\fG$ of elements of $M$ is a {\bf Gr\"obner basis} of $M$ if $\{ \init(\alpha) \mid \alpha \in \fG \}$ generates $\init(M)$. Note that $M$ has a finite Gr\"obner basis if and only if $\init(M)$ is finitely generated. As usual, we have:

\begin{proposition} \label{prop:GB-generates}
Let $\fG$ be a Gr\"obner basis of $M$. Then $\fG$ generates $M$.
\end{proposition}

\begin{proof}
Let $N \subseteq M$ be the subrepresentation generated by $\fG$. Then $\init(N)$ contains $\init(\alpha)$ for all $\alpha \in \fG$, and so $\init(N)=\init(M)$. Thus $M=N$ by Proposition~\ref{initeq}.
\end{proof}

We now come to our main result, which follows from the above discussion:

\begin{theorem} \label{grob-pp}
Suppose $\bk$ is left-noetherian, $S$ is orderable, and $\vert S \vert$ is noetherian. Then every subrepresentation of $P$ has a finite Gr\"obner basis. In particular, $P$ is a noetherian object of $\Rep_{\bk}(\cC)$.
\end{theorem}

\begin{remark} \label{rmk:Spairs}
We have not discussed the important topic of algorithms for Gr\"obner bases. See \cite[Chapter 15]{eisenbud} for an exposition of this theory for modules over polynomial rings. Two important results are Buchberger's criterion using S-pairs for determining if a set of elements is a Gr\"obner basis, and Schreyer's extension of this idea to calculate free resolutions. These ideas can be extended to our settings and will be developed in future work. 
\end{remark}

\subsection{Gr\"obner categories} \label{ss:grob-cat}

Let $\cC$ be an essentially small category. For an object $x$, let $S_x \colon \cC \to \Set$ be the functor given by $S_x(y)=\Hom(x, y)$. Note that $P_x=\bk[S_x]$.

\begin{definition} \label{def:grobner-cat}
We say that $\cC$ is {\bf Gr\"obner} if, for all objects $x$, the functor $S_x$ is orderable and the poset $\vert S_x \vert$ is noetherian. We say that $\cC$ is {\bf quasi-Gr\"obner} if there exists a Gr\"obner category $\cC'$ and an essentially surjective functor $\cC' \to \cC$ satisfying property~(F).
\end{definition}

The following theorem is one of the two main theoretical results of this paper. It connects the purely combinatorial condition ``(quasi-)Gr\"obner'' with the algebraic condition ``noetherian'' for representations.

\begin{theorem} \label{grobnoeth}
Let $\cC$ be a quasi-Gr\"obner category. Then for any left-noetherian ring $\bk$, the category $\Rep_{\bk}(\cC)$ is noetherian.
\end{theorem}

\begin{proof}
First suppose that $\cC$ is a Gr\"obner category. Then every principal projective of $\Rep_{\bk}(\cC)$ is noetherian, by Theorem~\ref{grob-pp}, and so $\Rep_{\bk}(\cC)$ is noetherian by Proposition~\ref{cat-noeth}.

Now suppose that $\cC$ is quasi-Gr\"obner, and let $\Phi \colon \cC' \to \cC$ be an essentially surjective functor satisfying property~(F), with $\cC'$ Gr\"obner. Then $\Rep_{\bk}(\cC')$ is noetherian, by the previous paragraph, and so $\Rep_{\bk}(\cC)$ is noetherian by Corollary~\ref{pullback}.
\end{proof}

\begin{remark}
If the functor $S_x$ is orderable, then the group $\Aut(x)$ admits a well-order compatible with the group operation, and is therefore trivial. Thus, in a Gr\"obner category, there are no non-trivial automorphisms.
\end{remark}

The definition of Gr\"obner is rather abstract. We now give a more concrete reformulation when $\cC$ is a directed category, which is the version we will apply in practice. Let $x$ be an object of $\cC$. An {\bf admissible order} on $\vert \cC_x \vert$ is a well-order $\preceq$ satisfying the following additional condition: given $f,f' \in \Hom(x,y)$ with $f \prec f'$ and $g \in \Hom(y,z)$, we have $gf \prec gf'$.

\begin{proposition} \label{prop:admissible-order}
If $\cC$ is a directed category, then $|\cC_x| \cong |S_x|$ as posets for all objects $x$. In particular, $\cC$ is Gr\"obner if and only if for all $x$ the set $\vert \cC_x \vert$ admits an admissible order and is noetherian as a poset {\rm (}conditions {\rm (G1)} and {\rm (G2)} from the introduction{\rm )}.
\end{proposition}

\begin{proof}
It suffices to treat the case where $\cC$ is small. Let $x$ be an object of $\cC$. The sets $\Ob(\cC_x)$ and $\wt{S}_x$ are equal: both are the set of all morphisms $x \to y$. In $\vert \cC_x \vert$, two morphisms $f$ and $g$ are identified if $g=hf$ for some isomorphism $h$. In $\vert S_x \vert$, two morphisms $f$ and $g$ are identified if there are morphisms $h$ and $h'$ such that $g=hf$ and $f=h'g$. Since $\cC$ is directed, $h$ and $h'$ are necessarily isomorphisms. So $\vert \cC_x \vert$ and $\vert S_x \vert$ are the same quotient of $\Ob(\cC_x)=\wt{S}_x$. The orders on each are defined in the same way, and thus the two are isomorphic posets. Thus $\vert \cC_x \vert$ is noetherian if and only if $\vert S_x \vert$ is.

Now let $\preceq$ be an admissible order on $\vert \cC_x \vert$. Since $\cC$ is directed, the natural map $S_x(y) \to \vert \cC_x \vert$ is an injection. We define a well-ordering on $S_x(y)$ by restricting $\preceq$ to it. One readily verifies that this defines an ordering of $S_x$.

Finally, suppose that $\preceq$ is an ordering on $S_x$. Let $\cC_0$ be a set of isomorphism class representatives for $\cC$. Since $\cC$ is directed, the natural map $\coprod_{y \in \cC_0} S_x(y) \to \vert \cC_x \vert$ is a bijection. Choose an arbitrary well-ordering $\preceq$ on $\cC_0$. Define a well-order $\preceq$ on $\coprod_{y \in \cC_0} S_x(y)$ as follows. If $f \colon x \to y$ and $g \colon x \to z$ then $f \preceq g$ if $y \prec z$, or $y=z$ and $f \prec g$ as elements of $S_x(y)$. One easily verifies that this induces an admissible order on $\vert \cC_x \vert$.
\end{proof}

The following two results follow easily from the definitions, so we omit their proofs.

\begin{proposition} \label{grobprod}
The cartesian product of finitely many {\rm (}quasi-{\rm )}Gr\"obner categories is {\rm (}quasi-{\rm )}Gr\"obner.
\end{proposition}

\begin{proposition} \label{potgrob}
Suppose that $\Phi \colon \cC' \to \cC$ is an essentially surjective functor satisfying property~{\rm (F)} and $\cC'$ is quasi-Gr\"obner. Then $\cC$ is quasi-Gr\"obner.
\end{proposition}

\subsection{Other properties} 
We end this section with some useful combinatorial properties.

\begin{definition} \label{propS}
A functor $\Phi \colon \cC' \to \cC$ satisfies {\bf property~(S)} (S for sub) if the following condition holds: if $f \colon x \to y$ and $g \colon x \to z$ are morphisms in $\cC'$ and there exists  $\wt{h} \colon \Phi(y) \to \Phi(z)$ such that $\Phi(g)=\wt{h} \Phi(f)$ then there exists $h \colon y \to z$ such that $g=hf$. A subcategory $\cC' \subset \cC$ satisfies property~(S) if the inclusion functor does.
\end{definition}

\begin{proposition} \label{grobsub}
Let $\Phi \colon \cC' \to \cC$ be a faithful functor satisfying property~{\rm (S)} and suppose $\cC$ is Gr\"obner. Then $\cC'$ is Gr\"obner.
\end{proposition}

\begin{proof}
Let $x$ be an object of $\cC'$. We first claim that the natural map $i \colon \vert S_x \vert \to \vert S_{\Phi(x)} \vert$ induced by $\Phi$ is strictly order-preserving. Indeed, let $f \colon x \to y$ and $g \colon x \to z$ be elements of $\vert S_x \vert$ such that $i(f) \le i(g)$. Then there exists $\wt{h} \colon \Phi(y) \to \Phi(z)$ such that $\wt{h} \Phi(f)=\Phi(g)$. By property~(S), there exists $h \colon y \to z$ such that $hf=g$. Thus $f \le g$, establishing the claim. It follows from this, and the noetherianity of $\vert S_{\Phi(x)} \vert$, that $\vert S_x \vert$ is noetherian. Finally, an ordering on $S_{\Phi(x)}$ obviously induces one on $S_{\Phi(x)} \vert_{\cC'}$, and this restricts to one on $S_x$. (Note that $S_x$ is a subfunctor of $S_{\Phi(x)} \vert_{\cC'}$ since $\Phi$ is faithful.)
\end{proof}

We end with a combinatorial criterion for $\Rep_\bk(\cC)$ to have Krull dimension $0$. Assume that $\bk$ is a field. Given $M \in \Rep_{\bk}(\cC)$, define $M^\vee \in \Rep_{\bk}(\cC^{\op})$ by $M^\vee(x) = \hom_\bk(M(x), \bk)$.

\begin{definition} \label{defn:propD}
We say that $\cC$ satisfies {\bf property~(D)} (D for dual) if the following condition holds: for every object $x$ of $\cC$ there exist finitely many objects $\{y_i\}_{i \in I}$ and finite subsets $S_i$ of $\Hom(x,y_i)$ such that if $f \colon x \to z$ is any morphism then there exists $i \in I$, a morphism $g \colon z \to y_i$, and $\sigma_i \in S_i$ such that $g_*^{-1}(\{\sigma_i\})=\{f\}$, where $g_*$ is the map $\Hom(x, z) \to \Hom(x, y_i)$.
\end{definition}

\begin{proposition} \label{prop:propD}
Suppose that $\cC$ is $\Hom$-finite and satisfies property~{\rm (D)}, and $\bk$ is a field. Then for any object $x$, the representation $P_x^{\vee}$ of $\cC^{\op}$ is finitely generated.
\end{proposition}

\begin{proof}
The space $P_x(y)$ has a basis $\{e_f\}$. We let $e_f^*$ be the dual basis of $P_x^{\vee}(y)$. We note that the $e_f^*$ generate $P_x^{\vee}$ as a representation. Let $y_i$ and $S_i$ be as in the definition of property~(D). As we range over $i \in I$ and $\sigma \in S_i$, the elements $e_{\sigma}^* \in \bk[\Hom(x,y_i)]$ generate $P_x^\vee$.
\end{proof}

\begin{proposition} \label{dim0}
Assume that $\bk$ is a field. Suppose that $\cC$ is $\Hom$-finite and satisfies property~{\rm (D)}, and that $\cC$ and $\cC^{\op}$ are both quasi-Gr\"obner. Then $\Rep_{\bk}(\cC)$ has Krull dimension $0$, that is, every finitely generated representation of $\cC$ has finite length.
\end{proposition}

\begin{proof} 
The property of having finite length is preserved under finite direct sums and quotients, so it suffices to prove that each principal projective $P$ has finite length. By Proposition~\ref{prop:propD}, $P^\vee$ is finitely generated. Since $\cC$ and $\cC^\op$ are quasi-Gr\"obner, both $P$ and $P^\vee$ satisfy the ascending chain condition (Theorem~\ref{grobnoeth}). So $P$ also satisfies the descending chain condition and has finite length.
\end{proof}

\section{Formal languages}
\label{s:lang}

In this section we give basic definitions and results on formal languages. In particular, we define a few classes of formal languages (regular, ordered, unambiguous context-free) which will be used in this paper along with results on their generating functions. We believe that the material in \S\ref{ss:ordered} on ordered languages is new. The rest of the material is standard.

The reader can consult \cite{hopcroftullman} for a general reference on formal languages.

\subsection{Generalities} \label{sec:lang-gen}

Fix a finite set $\Sigma$ (which we also call an {\bf alphabet}). A {\bf language} on $\Sigma$ is a subset of $\Sigma^{\star}$. Let $\cL$ and $\cL'$ be two languages. The {\bf union} of $\cL$ and $\cL'$, denoted $\cL \cup \cL'$, is their union as subsets of $\Sigma^{\star}$. The {\bf concatenation} of $\cL$ and $\cL'$ is $\cL \cL' = \{ww' \mid w \in \cL,\ w' \in \cL'\}$. The {\bf Kleene star} of $\cL$ is $\cL^{\star} = \{w_1\cdots w_n \mid w_1,\ldots,w_n \in \cL\}$, i.e., the submonoid (under concatenation) of $\Sigma^{\star}$ generated by $\cL$.

Let $\Sigma$ be an alphabet. A {\bf norm} on $\Sigma$ is a monoid homomorphism $\nu \colon \Sigma^{\star} \to \bN^I$, for some finite set $I$ which is induced by a function $\Sigma \to I$.
A norm on a language $\cL$ over $\Sigma$ is a function $\cL \to \bN^I$ which is the restriction of a norm on $\Sigma$. Let $\bt=(t_i)_{i \in I}$ be indeterminates. The {\bf Hilbert series} of $\cL$ (with respect to $\nu$) is
\begin{displaymath}
\rH_{\cL, \nu}(\bt) = \sum_{w \in \cL} \bt^{\nu(w)},
\end{displaymath}
when this makes sense (i.e., the coefficient of $\bt^{\bn}$ is finite for all $\bn$). The coefficient of $\bt^{\bn}$ in $\rH_{\cL, \nu}(\bt)$ is the number of words of norm $\bn$. We omit  $\nu$ from the notation when possible.

Every language admits a canonical norm over $\bN$, namely, the length function $\ell \colon \cL \to \bN$. The coefficient of $t^n$ in $\rH_{\cL, \ell}$ counts the number of words of length $n$ in $\cL$. This series is the {\bf generating function} of $\cL$ in the literature. 

We say that a norm $\nu$ is {\bf universal} if the map $\Sigma \to I$ is injective. We say that the corresponding Hilbert series is {\bf universal}. Any Hilbert series of $\cL$ (under any norm) can be expressed as the image of a universal Hilbert series of $\cL$ under a homomorphism $\bN^I \to \bN^J$.

\subsection{Regular languages}

The set of {\bf regular languages} on $\Sigma$ is the smallest set of languages on $\Sigma$ containing the empty language and the singleton languages $\{c\}$ (for $c \in \Sigma$), and closed under finite union, concatenation, and Kleene star. A {\bf deterministic finite-state automata} (DFA) for $\Sigma$ is a tuple $(Q, T, \sigma, \cF)$ consisting of:
\begin{itemize}
\item A finite set $Q$, the set of {\bf states}.
\item A function $T \colon \Sigma \times Q \to Q$, the {\bf transition table}.
\item A state $\sigma \in Q$, the {\bf initial state}.
\item A subset $\cF \subset Q$, the set of {\bf final states}.
\end{itemize}
Fix a DFA. We write $\alpha \stackrel{c}{\to} \beta$ to indicate $T(c,\alpha)=\beta$. Given a word $w=w_1\cdots w_n$, we write $\alpha \stackrel{w}{\to} \beta$ if there are states $\alpha=\alpha_0,\ldots,\alpha_n=\beta$ such that
\begin{displaymath}
\alpha_0 \stackrel{w_1}{\to} \alpha_1 \stackrel{w_2}{\to} \cdots \stackrel{w_{n-1}}{\to} \alpha_{n-1} \stackrel{w_n}{\to} \alpha_n.
\end{displaymath}
Intuitively, we think of $\alpha \stackrel{w}{\to} \beta$ as saying that the DFA starts in state $\alpha$, reads the word $w$, and ends in the state $\beta$. We say that the DFA {\bf accepts} the word $w$ if $\sigma \stackrel{w}{\to} \tau$ with $\tau$ a final state. The set of accepted words is called the language {\bf recognized} by the DFA.
The following results are standard (see \cite[Ch.~2]{hopcroftullman} and \cite[Theorem 4.7.2]{stanley-EC1}, for example).

\begin{theorem} 
A language is regular if and only if it is recognized by some DFA.
\end{theorem}

\begin{theorem} \label{reghilb} 
If $\cL$ is regular then $\rH_{\cL, \nu}(\bt)$ is a rational function of $\bt$, for any norm $\nu$.
\end{theorem}

\subsection{Ordered languages}
\label{ss:ordered}

The set of {\bf ordered languages} on $\Sigma$ is the smallest set of languages on $\Sigma$ that contains the singleton languages and the languages $\Pi^{\star}$, for $\Pi \subseteq \Sigma$, and that is closed under finite union and concatenation. We have not found this class of languages considered in the literature.

A DFA is {\bf ordered} if the following condition holds: if $\alpha$ and $\beta$ are states and there exist words $u$ and $w$ with $\alpha \stackrel{u}{\to} \beta$ and $\beta \stackrel{w}{\to} \alpha$, then $\alpha = \beta$. The states of an ordered DFA admit a natural partial order by $\alpha \le \beta$ if there exists a word $w$ with $\alpha \stackrel{w}{\to} \beta$.

We prove two theorems about ordered languages. The first is the following.

\begin{theorem}
A language is ordered if and only if it is recognized by some ordered DFA.
\end{theorem}

\begin{proof}
If $\cL$ is an ordered language then we can write $\cL=\cL_1 \cup \cdots \cup \cL_n$, where each $\cL_i$ is a concatenation of singleton languages and languages of the form $\Pi^{\star}$ with $\Pi \subset \Sigma$. By Lemmas~\ref{ord-concat1} and~\ref{ord-concat2} below, each $\cL_i$ is recognized by an ordered DFA. Thus by Lemma~\ref{ord-union} below, $\cL$ is recognized by an ordered DFA.
The converse is proved in Lemma~\ref{ord-conv}.
\end{proof}

\begin{lemma} \label{ord-union}
Suppose that $\cL$ and $\cL'$ are languages recognized by ordered DFA's. Then $\cL \cup \cL'$ is recognized by an ordered DFA.
\end{lemma}

\begin{proof}
Suppose $(Q, T, \sigma, \cF)$ is an ordered DFA recognizing $\cL$ and $(Q', T', \sigma', \cF')$ is an ordered DFA recognizing $\cL'$. Let $Q'' = Q \times Q'$. Define $T'' \colon \Sigma \times Q'' \to Q''$ by
\begin{displaymath}
T''(c, (\alpha, \alpha'))=(T(c,\alpha),T'(c,\alpha')).
\end{displaymath}
Let $\sigma''=(\sigma,\sigma')$, and $\cF''=(\cF \times Q') \cup (Q \times \cF')$. Then $(Q'', T'', \sigma'', \cF'')$ is a DFA recognizing $\cL \cup \cL'$. Since $(\alpha,\alpha') \stackrel{w}{\to} (\beta,\beta')$ if and only if $\alpha \stackrel{w}{\to} \beta$ and $\alpha' \stackrel{w}{\to} \beta'$, this DFA is ordered.
\end{proof}

\begin{lemma} \label{ord-uniq-final}
If $\cL$ is a non-empty language recognized by an ordered DFA then we can write $\cL$ as a finite union of languages recognized by an ordered DFA with a single final state.
\end{lemma}

\begin{proof}
Suppose $\cL$ is recognized by the ordered DFA $(Q, T, \sigma, \cF)$. Enumerate $\cF$ as $\{\tau_1, \ldots, \tau_n\}$. Let $\cL_i$ be the set of words $w \in \cL$ for which $\sigma \stackrel{w}{\to} \tau_i$. Then $\cL$ is clearly the union of the $\cL_i$. Moreover, $\cL_i$ is recognized by the ordered DFA $(Q, T, \sigma, \{\tau_i\})$.
\end{proof}

\begin{lemma} \label{ord-concat1}
Let $\cL$ be a language recognized by an ordered DFA and let $\cL'=\{c\}$ be a singleton language. Then the concatentation $\cL \cL'$ is recognized by an ordered DFA.
\end{lemma}

\begin{proof}
Suppose first that $\cL$ is recognized by the ordered DFA $(Q, T, \sigma, \cF)$, where $\cF=\{\tau\}$ has a single element. If $T(c,\tau)=\tau$, then nothing needs to be changed since $\cL=\cL \cL'$. Suppose then that $T(c,\tau)= \rho \ne \tau$. Let $Q'=Q \amalg \{\tau'\}$, and define a transition function $T'$ as follows. First, $T'(b, \alpha)=T(b,\alpha)$ if $\alpha \in Q$, unless $\alpha=\tau$ and $b=c$. We define $T'(c,\tau)=\tau'$. Finally, we define $T'(b,\tau')=\rho$, for any $b$. Then $(Q', T', \sigma, \{\tau'\})$ is an ordered DFA recognizing $\cL \cL'$. For the general case, use Lemmas~\ref{ord-union} and \ref{ord-uniq-final}.
\end{proof}

\begin{lemma} \label{ord-concat2}
Let $\cL$ be a language recognized by an ordered DFA and let $\cL'=\Pi^{\star}$ for some subset $\Pi$ of $\Sigma$. Then the concatentation $\cL \cL'$ is recognized by an ordered DFA.
\end{lemma}

\begin{proof}
Suppose first that $\cL$ is accepted by the ordered DFA $(Q, T, \sigma, \cF)$, where $\cF=\{\tau\}$. Let $Q'=Q \amalg \{\tau',\rho\}$, and define a transition function $T'$ as follows. First, $T'(c,\alpha)=T(c,\alpha)$ for $\alpha \in Q \setminus \{\tau\}$. Let $\Delta$ be the set of elements $c \in \Sigma$ such that $T(c,\tau)=\tau$. We define
\begin{displaymath}
T'(c,\tau) = \begin{cases}
\tau & \text{if } c \in \Delta \\
\tau' & \text{if } c \in \Pi \setminus \Delta \\
\rho & \text{if } c \not\in \Pi \cup \Delta
\end{cases}, \qquad
T'(c,\tau') = \begin{cases}
\tau' & \text{if } c \in \Pi \\
\rho & \text{if } c \not\in \Pi
\end{cases}.
\end{displaymath}
Finally, we define $T'(c,\rho)=\rho$ for all $c \in \Sigma$. One easily verifies that $(Q', T', \sigma, \{\tau, \tau'\})$ is an ordered DFA accepting $\cL \cL'$. The deduction of the general case from this special case proceeds exactly as the corresponding argument in the proof of Lemma~\ref{ord-concat1}.
\end{proof}

\begin{lemma} \label{ord-conv}
The language recognized by an ordered automata is ordered.
\end{lemma}

\begin{proof}
Let $(Q,T,\sigma,\cF)$ be an ordered automata, and let $\cL$ be the language it recognizes. We show that $\cL$ is ordered, following the proof of \cite[Thm.~2.4]{hopcroftullman}. Enumerate the states $Q$ as $\{\alpha_1, \ldots, \alpha_s\}$ in such a way that if $\alpha_i \le \alpha_j$ then $i \le j$. Let $\cL_{i,j}^k$ be the set of words $w=w_1 \cdots w_n$ such that
\begin{displaymath}
\alpha_i=\beta_0 \stackrel{w_1}{\to} \beta_1 \stackrel{w_2}{\to} \cdots \stackrel{w_n}{\to} \beta_n = \alpha_j \quad \text{ with } \quad \beta_1, \ldots, \beta_{n-1} \in \{\alpha_1, \ldots, \alpha_k \}.
\end{displaymath}
In other words, $w \in \cL_{i,j}^k$ if it induces a transition from $\alpha_i$ to $\alpha_j$ via intermediate states of the form $\alpha_{\ell}$ with $\ell \le k$. We prove by induction on $k$ that each $\cL_{i,j}^k$ is an ordered language. To begin with, when $k=0$ no intermediate states are allowed (i.e., $n=1$), and so $\cL_{i,j}^k$ is a subset of $\Sigma$, and therefore an ordered language. For $k \ge 1$, we have
\begin{displaymath}
\cL_{i,j}^k = \cL_{i,k}^{k-1} (\cL_{k,k}^{k-1})^{\star} \cL_{k,j}^{k-1} \cup \cL_{i,j}^{k-1}.
\end{displaymath}
Since there is no way to transition from $\alpha_k$ to $\alpha_i$ with $i<k$, any word in $\cL_{k,k}^{k-1}$ must have length $1$. Thus $\cL_{k,k}^{k-1}$ is a subset of $\Sigma$, and so $(\cL_{k,k}^{k-1})^{\star}$ is an ordered language. The above formula then establishes inductively that $\cL_{i,j}^k$ is ordered for all $k$.

Let $\cF=\{\alpha_j\}_{j \in J}$ be the set of final states, and let $\sigma=\alpha_i$ be the initial state. Then $\cL=\bigcup_{j \in J} \cL_{i,j}^s$, and is therefore an ordered language.
\end{proof}

The following is our second main result about ordered languages. To state it, we need one piece of terminology: we say that a subset $\Pi$ of $\Sigma$ is {\bf repeatable} with respect to some language $\cL$ if there exist $w$ and $w'$ in $\Sigma^{\star}$ such that $w \Pi^{\star} w' \subset \cL$.

\begin{theorem} \label{ordhilb}
Suppose $\cL$ is an ordered language with a norm $\nu$. Let $\Pi_1, \ldots, \Pi_r$ be the repeatable subsets of $\Sigma$, and let $\lambda_i = \sum_{c \in \Pi_i} \bt^{\nu(c)}$. Then $\rH_{\cL}(\bt) = f(\bt)/g(\bt)$ where $f(\bt)$ and $g(\bt)$ are polynomials, and $g(\bt)$ factors as $\prod_{i=1}^r(1-\lambda_i)^{e_i}$ where $e_i \ge 0$.
\end{theorem}

\begin{proof}
Choose an ordered DFA recognizing $\cL$. We can assume that every state is $\ge$ the initial state. Call a state $\alpha$ {\bf prefinal} if there exists a final state $\tau$ such that $\alpha \le \tau$.  Enumerate the states as $\alpha_1, \ldots, \alpha_s$ such that the following conditions hold: (1) $\alpha_1$ is the initial state; (2) if $\alpha_i \le \alpha_j$ then $i \le j$; and (3) there exists $n$ such that $\{\alpha_1, \ldots, \alpha_n\}$ is the set of prefinal states. Let $A$ be the $s \times s$ matrix $A_{i,j}=\sum_{\substack{\alpha_i \stackrel{w}{\to} \alpha_j \\ \ell(w)=1}} \bt^{\nu(w)}$; then $(A^k)_{i,j}=\sum_{\substack{\alpha_i \stackrel{w}{\to} \alpha_j \\ \ell(w)=k}} \bt^{\nu(w)}$.
Then $A$ is upper-triangular and the $i$th diagonal entry is $\sum_{c \in \Pi} \bt^{\nu(c)}$, where $\Pi$ is the set of letters $c$ which induce a transition from $\alpha_i$ to itself. If $\alpha_i$ is prefinal then $\Pi$ is one of the $\Pi_j$'s, and so the diagonal entry at $(i,i)$ is one of the $\lambda_j$'s. If $\alpha_{t_1}, \ldots, \alpha_{t_s}$ are the final states then 
\begin{displaymath}
\rH_{\cL}(\bt) = \sum_{i=0}^s \sum_{k \ge 0} (A^k)_{1,t_i}= \sum_{i=1}^s \frac{(-1)^{1 + t_i} \det(1-A : t_i, 1)}{\det(1-A)},
\end{displaymath}
where the notation $(1-A: t_i, 1)$ means that we remove the $t_i$th row and first column from $1-A$. It is clear that this is of the stated form.
\end{proof}

\begin{corollary} \label{ordhilbcor}
Suppose $\cL$ is an ordered language. Then $\rH_{\cL,\ell}(t)$ can be written in the form $f(t)/g(t)$  where $f(t)$ and $g(t)$ are polynomials, and $g(t)$ factors as $\prod_{a=1}^r (1-at)^{e(a)}$ where $e(a) \ge 0$ and $r$ is the cardinality of the largest repeatable subset of $\Sigma$ with respect to $\cL$.
\end{corollary}

\subsection{Unambiguous context-free languages}

In this paper, this section will only be used in \S\ref{sec:nonregular-lang}. Let $\Sigma$ be a finite alphabet, which we also call {\bf terminal} symbols. Let $N$ be another finite set, disjoint from $\Sigma$, which we call the {\bf non-terminal} symbols. A {\bf production rule} is $n \to P(n)$ where $n \in N$ and $P(n)$ is a word in $\Sigma \cup N$. A {\bf context-free grammar} is a tuple $(\Sigma, N, \cP, n_0)$, where $\Sigma$ and $N$ are as above, $\cP$ is a finite set of production rules, and $n_0$ is a distinguished element of $N$. Given such a grammar, a word $w$ in $\Sigma$ is a {\bf valid $n$-expression} if there is a production rule $n \to P(n)$, where $P(n)=c_1 \cdots c_r$, and a decomposition $w=w_1 \cdots w_r$ (with $r>1$), such that $w_i$ is a valid $c_i$-expression if $c_i$ is non-terminal, and $w_i=c_i$ if $c_i$ is terminal. The language {\bf recognized} by a grammar is the set of valid $n_0$-expressions, and a language of this form is called a {\bf context-free language}. A grammar is {\bf unambiguous} if for each valid $n$-expression $w$ the production rule $n \to P(n)$ and decomposition of $w$ above is unique. An {\bf unambiguous context-free language} is one defined by such a grammar. See \cite[Definition 6.6.4]{stanley-EC2} for an alternative description.

Given a coefficient field $K$, a formal power series $F(t) \in K[\![t]\!]$ is {\bf algebraic} if there exist polynomials $p_0(t), \dots, p_d(t) \in K[t]$ such that $\sum_{i=0}^d p_i(t) F(t)^i = 0$. See \cite[\S 6.1]{stanley-EC2} for more information.

\begin{theorem}[Chomsky--Sch\"utzenberger \cite{chomsky}] \label{thm:CS-UCF}
Let $\cL$ be an unambiguous context-free language equipped with a norm $\nu$. Then $\rH_{\cL,\nu}(\bt)$ is an algebraic function.
\end{theorem}

\begin{proof}
Let $\cL$ be a context-free language and pick a grammar that recognizes $\cL$. Consider the modified Hilbert series $\sum_{w \in \cL} a_w \bt^{\nu(w)}$ where $\nu$ is a universal norm and $a_w$ is the number of ways that $w$ can be built using the production rules. This is an algebraic function \cite[Theorem 6.6.10]{stanley-EC2}. If $\cL$ is unambiguous, then we can pick a grammar so that $a_w=1$ for all $w \in \cL$, and we obtain the usual Hilbert series.
\end{proof}

\section{Hilbert series and lingual categories}
\label{sec:ling}

In this section we use the theory of formal languages (discussed in \S\ref{s:lang}) together with the Gr\"obner techniques in \S\ref{sec:noeth-grob} to study Hilbert series of representations of quasi-Gr\"obner categories. Our goal is to introduce norms, Hilbert series, lingual structures, and lingual categories and to prove a few basic properties about lingual categories.

\subsection{Normed categories and Hilbert series}

Let $\cC$ be an essentially small category. A {\bf norm} on $\cC$ is a function $\nu \colon \vert \cC \vert \to \bN^I$, where $I$ is a finite set. A {\bf normed category} is a category equipped with a norm. Fix a category $\cC$ with a norm $\nu$ with values in $\bN^I$. As in \S\ref{sec:lang-gen}, we let $\{t_i\}_{i \in I}$ be indeterminates. Let $M$ be a representation of $\cC$ over a field $\bk$. We define the {\bf Hilbert series} of $M$ as
\begin{displaymath}
\rH_{M, \nu}(\bt) = \sum_{x \in \vert \cC \vert} \dim_{\bk}{M(x)} \cdot \bt^{\nu(x)},
\end{displaymath}
when this makes sense. We omit the norm $\nu$ from the notation when possible.

\subsection{Lingual structures}

We return to the set-up of \S \ref{ss:monomial}: let $S \colon \cC \to \Set$ be a functor, and let $P=\bk[S]$. However, we now also assume that $\cC$ is directed and normed over $\bN^I$. We define a norm on $\vert S \vert$ as follows: given $f \in \vert S \vert$, let $\wt{f} \in S(x)$ be a lift, and put $\nu(f)=\nu(x)$. This is well-defined because $\cC$ is ordered: if $\wt{f}' \in S(y)$ is a second lift then $x$ and $y$ are necessarily isomorphic. Let $\cP$ be a class of languages (e.g., regular languages).

\begin{definition}
A {\bf lingual structure} on $\vert S \vert$ is a pair $(\Sigma, i)$ consisting of a finite alphabet $\Sigma$ normed over $\bN^I$ and an injection $i \colon \vert S \vert \to \Sigma^{\star}$ compatible with the norms, i.e., such that $\nu(i(f))=\nu(f)$. A {\bf $\cP$-lingual structure} is a lingual structure satisfying the following additional condition: for every poset ideal $J$ of $\vert S \vert$, the language $i(J)$ is of class $\cP$.
\end{definition}

\begin{theorem} \label{lingthm}
Suppose $\cC$ is directed, $S$ is orderable, $\vert S \vert$ is noetherian, and $\vert S \vert$ admits a $\cP$-lingual structure. Let $M$ be a subrepresentation of $P$. Then $\rH_M(\bt)$ is of the form $\rH_{\cL}(\bt)$, where $\cL$ is a language of class $\cP$.
\end{theorem}

\begin{proof}
Choose an ordering on $S$ and let $N$ be the initial representation of $M$. Given $f \in S(x)$, let $F^f P(x)$ be the span of the elements $e_g$ with $g \le f$. Then $F^{\bullet} P(x)$ is a filtration on $P(x)$ and induces one on $M(x)$. The associated graded of $M(x)$ is exactly $N(x)$, and so $M(x)$ and $N(x)$ have the same dimension. Thus $M$ and $N$ have the same Hilbert series. Let $J \subset \vert S \vert$ be the set of elements $f$ for which $e_f$ belongs to $N$. Then $N$ and the language $i(J)$ have the same Hilbert series.
\end{proof}

\begin{remark}
We abbreviate ``ordered,'' ``regular,'' and ``unambiguous context-free'' to O, R, and UCF. So an R-lingual structure is one for which $i(I)$ is always regular.
\end{remark}

\begin{remark}
If $\vert S \vert$ admits a lingual structure then $S(x)$ is finite for all $x$. Indeed, given $\bn \in \bN^I$, let $\Sigma^\star_{\bn}$ denote the set of words of norm $\bn$; this is a finite set since all such words have length $\vert \bn \vert$. Since $i$ maps $S(x)$ injectively into $\Sigma^{\star}_{\nu(x)}$, we see that $S(x)$ is finite.
\end{remark}

\subsection{Lingual categories} \label{ss:lingual-cat}

Recall that $S_x$ denotes the functor $\Hom_\cC(x, -)$.

\begin{definition}
A directed normed category $\cC$ is {\bf $\cP$-lingual} if $\vert S_x \vert$ admits a $\cP$-lingual structure for all objects $x$.
\end{definition}

The following theorem is the second main theoretical result of this paper. It connects the purely combinatorial condition ``$\cP$-lingual'' with the algebraic invariant ``Hilbert series.''

\begin{theorem} \label{hilbthm}
Let $\cC$ be a $\cP$-lingual Gr\"obner category and let $M$ be a finitely generated representation of $\cC$. Then $\rH_M(\bt)$ is a $\bZ$-linear combination of series of the form $\rH_{\cL}(\bt)$ where $\cL$ is a language of class $\cP$. In particular,
\begin{itemize}
\item If $\cP=\mathrm{R}$, then $\rH_M(\bt)$ is a rational function of the $t_i$.
\item If $\cP=\mathrm{O}$, then $\rH_M(\bt) = f(\bt)/g(\bt)$, where $f(\bt)$ is a polynomial in the $t_i$ and $g(\bt) = \prod_{j=1}^n (1-\lambda_j)$ and each $\lambda_j$ is an $\bN$-linear combination of the $t_i$.
\item If $\cP=\mathrm{UCF}$, then $\rH_M(\bt)$ is an algebraic function of the $t_i$.
\end{itemize}
\end{theorem}

\begin{proof}
If $M$ is a subrepresentation of a principal projective then the first statement follows from Theorem~\ref{lingthm}. As such representations span the Grothendieck group of finitely generated representations, the first statement holds.

The remaining statements follow from the various results about Hilbert series of $\cP$-languages: in the regular case, Theorem~\ref{reghilb}; in the ordered case, Theorem~\ref{ordhilb}; and in the unambiguous context-free case, the relevant result is Theorem~\ref{thm:CS-UCF}.
\end{proof}

Suppose $\cC_1$ is normed over $\bN^{I_1}$ and $\cC_2$ is normed over $\bN^{I_2}$. We give $\cC=\cC_1 \times \cC_2$ the structure of a normed category over $\bN^I$, where $I=I_1 \amalg I_2$, by $\nu(x,y)=\nu(x)+\nu(y)$. Here we have identified $\vert \cC \vert$ with $\vert \cC_1 \vert \times \vert \cC_2 \vert$. We call $\cC$, with this norm, the {\bf product} of $\cC_1$ and $\cC_2$.

\begin{proposition} \label{prodling}
Let $\cC_1$ and $\cC_2$ be $\cP$-lingual normed categories. Suppose that the posets $\vert \cC_{1,x} \vert$ and $\vert \cC_{2,y} \vert$ are noetherian for all $x$ and $y$ and that $\cP$ is stable under finite unions and concatenations of languages on disjoint alphabets. Then $\cC_1 \times \cC_2$ is also $\cP$-lingual.
\end{proposition}

\begin{proof}
Keep the notation from the previous paragraph. Let $x$ be an object of $\cC_1$, and let $i_1 \colon \vert \cC_{1,x} \vert \to \Sigma_1^{\star}$ be a $\cP$-lingual structure at $x$. Similarly, let $y$ be an object of $\cC_2$, and let $i_2 \colon \vert \cC_{2,y} \vert \to \Sigma_2^{\star}$ be a $\cP$-lingual structure at $y$. Let $\Sigma=\Sigma_1 \amalg \Sigma_2$, normed over $\bN^I$ in the obvious manner. Then the following diagram commutes:
\begin{displaymath}
\xymatrix{
\vert \cC_{(x,y)} \vert \ar@{=}[r] &
\vert \cC_{1,x} \vert \times \vert \cC_{2,y} \vert \ar[r]^-{i_1 \times i_2} \ar[d] & \Sigma_1^{\star} \times \Sigma_2^{\star} \ar[r] \ar[d] & \Sigma^{\star} \ar[d] \\
& \bN^{I_1} \oplus \bN^{I_2} \ar@{=}[r] & \bN^{I_1} \oplus \bN^{I_2} \ar@{=}[r] & \bN^I }
\end{displaymath}
The top right map is concatenation of words. We let $i \colon \vert \cC_{(x,y)} \vert \to \Sigma^*$ be the composition along the first line, which is clearly injective. We claim that this is a $\cP$-lingual structure on $\vert \cC_{(x,y)} \vert$. The commutativity of the above diagram shows that it is a lingual structure. Now suppose $S$ is an ideal of $\vert \cC_{(x,y)} \vert$. Since this poset is noetherian (Proposition~\ref{noethprod}), it is a finite union of principal ideals $S_1, \ldots, S_n$. Each $S_j$ is of the form $T_j \times T_j'$, where $T_j$ is an ideal of $\cC_{1,x}$ and $T_j'$ is an ideal of $\cC_{2,y}$. By assumption, the languages $i_1(T_j)$ and $i_2(T'_j)$ satisfy property $\cP$. Regarding $i_1(T_j)$ and $i_2(T'_j)$ as languages over $\Sigma$, the language $i(S_j)$ is their concatentation, and therefore satisfies $\cP$. Finally, $i(S)$ satisfies $\cP$, as it is the union of the $i(S_j)$.
\end{proof}

\part{Applications}

\section{Categories of injections} \label{sec:cat-inj}

In this first applications section, we study the category $\FI$ of finite sets and injective functions along with generalizations and variations. The main results are listed in \S\ref{sec:OIFI-defn}, applications to twisted commutative algebras are in \S\ref{ss:tca}, and complementary results are listed in \S\ref{sec:FI-additional}.

\subsection{The categories $\OI_d$ and $\FI_d$} \label{sec:OIFI-defn}

Let $d$ be a positive integer. Define $\FI_d$ to be the following category. The objects are finite sets. Given two finite sets $S$ and $T$, a morphism $S \to T$ is a pair $(f,g)$ where $f \colon S \to T$ is an injection and $g \colon T \setminus f(S) \to \{1,\ldots, d\}$ is a function. Define $\OI_d$ to be the ordered version of $\FI_d$: its objects are totally ordered finite sets and its morphisms are pairs $(f,g)$ with $f$ an order-preserving injection (no condition is placed on $g$). We norm $\FI_d$ and $\OI_d$ over $\bN$ by $\nu(x)=\# x$. When $d=1$, we will write $\FI$ and $\OI$ instead of $\FI_1$ and $\OI_1$.

Let $\Sigma=\{0,\ldots,d\}$, and let $\cL$ be the language on $\Sigma$ consisting of words $w_1 \cdots w_r$ in which exactly $n$ of the $w_i$ are equal to $0$. Partially order $\cL$ using the subsequence order, i.e., if $s \colon [i] \to \Sigma$ and $t \colon [j] \to \Sigma$ are words then $s \le t$ if there exists $I \subseteq [j]$ such that $s=t \vert_I$.

\begin{lemma} \label{lem:oi-ordered}
The poset $\cL$ is noetherian and every ideal is an ordered language.
\end{lemma}

\begin{proof}
Noetherianity is an immediate consequence of Higman's lemma (Theorem~\ref{thm:higman}). Let $w=w_1\cdots w_n$ be a word in $\cL$. Then the ideal generated by $w$ is the language
\begin{displaymath}
\Pi^{\star} w_1 \Pi^{\star} w_2 \cdots \Pi^{\star} w_n \Pi^{\star},
\end{displaymath}
where $\Pi=\Sigma \setminus \{0\}$, and is therefore ordered. As every ideal is a finite union of principal ideals, and a finite union of ordered languages is ordered, the result follows.
\end{proof}

Our main result about $\OI_d$ is the following theorem.

\begin{theorem} \label{oithm}
The category $\OI_d$ is $\rO$-lingual and Gr\"obner.
\end{theorem}

\begin{proof}
It is clear that $\cC=\OI_d$ is directed. Let $n$ be a non-negative integer, and regard $x=[n]$ as an object of $\cC$. 

Pick $(f,g) \in \Hom_\cC([n], [m])$. Define $h \colon [m] \to \Sigma$ to be the function which is 0 on the image of $f$, and equal to $g$ on the complement of the image of $f$. One can recover $(f,g)$ from $h$ since $f$ is required to be order-preserving and injective. This construction therefore defines an isomorphism of posets $i \colon \vert \cC_x \vert \to \cL$. It follows that $\vert \cC_x \vert$ is noetherian. Furthermore, the lexicographic order on $\cL$ (using the standard order on $\Sigma$) is easily verified to restrict to an admissible order on $\vert \cC_x \vert$. Thus $\cC$ is Gr\"obner. 

For $f \in \vert \cC_x \vert$ we have $\vert f \vert=\ell(i(f))$, and so $(\Sigma, i)$ is a lingual structure on $\vert \cC_x \vert$ if we norm words in $\Sigma^\star$ by their length. Finally, an ideal of $\vert \cC_x \vert$ gives an ordered language over $\Sigma$, and so this is an O-lingual structure.
\end{proof}

\begin{remark} \label{rmk:OI1-poly}
The results about $\OI$ can be made more transparent with the following observation: the set of order-increasing injections $f \colon [n] \to [m]$ is naturally in bijection with monomials in $x_0, \dots, x_n$ of degree $m-n$ by assigning the monomial $\bm_f = \prod_{i=0}^n x_i^{f(i+1) - f(i) - 1}$ using the convention $f(0)=0$ and $f(n+1)=m+1$. Given $g \colon [n] \to [m']$, there is a morphism $h \colon [m] \to [m']$ with $g = hf$ if and only if $\bm_f$ divides $\bm_g$. Thus the monomial subrepresentations of $P_n$ are in bijection with monomial ideals in the polynomial ring $\bk[x_0,\dots,x_n]$. 
\end{remark}

\begin{theorem} \label{fithm}
The forgetful functor $\Phi \colon \OI_d \to \FI_d$ satisfies property~{\rm (F)}. In particular, $\FI_d$ is quasi-Gr\"obner.
\end{theorem}

\begin{proof}
Let $x=[n]$ be a given object of $\FI_d$. If $y$ is any totally ordered set, then any morphism $f \colon x \to y$ can be factored as $x \stackrel{\sigma}{\to} x \stackrel{f'}{\to} y$, where $\sigma$ is a permutation and $f'$ is order-preserving. It follows that we can take $y_1, \ldots, y_{n!}$ to all be $[n]$, and $f_i \colon x \to \Phi(y_i)$ to be the $i$th element of the symmetric group $S_n$ (under any enumeration). This establishes the claim. Since $\OI_d$ is Gr\"obner, this shows that $\FI_d$ is quasi-Gr\"obner.
\end{proof}

\begin{corollary} \label{fi-noeth}
If $\bk$ is left-noetherian then $\Rep_{\bk}(\FI_d)$ is noetherian.
\end{corollary}

\begin{remark} \label{rmk:fi-noeth}
This result was first proved for $\bk$ a field of characteristic $0$ in \cite[Thm.~2.3]{delta-mod}, though the essential idea goes back to Weyl; see also \cite[Proposition 9.2.1]{expos}. It was reproved independently for $d=1$ (and $\bk$ a field of characteristic $0$) in \cite{fimodules}. It was then proved for $d=1$ and $\bk$ an arbitrary ring in \cite{fi-noeth}. The result is new if $d>1$ and $\bk$ is not a field of characteristic $0$. 
\end{remark}

\begin{corollary} \label{fihilb}
Let $M$ be a finitely generated $\FI_d$-module over a field $\bk$. The Hilbert series
\begin{displaymath}
\rH_M(t) = \sum_{n \ge 0} \dim_\bk{M([n])} \cdot t^n
\end{displaymath}
is of the form $f(t)/g(t)$, where $f(t)$ and $g(t)$ are polynomials and $g(t)$ factors as $\prod_{j=1}^d (1-jt)^{e_j}$ where $e_j \ge 0$. In particular, if $d=1$ then $n \mapsto \dim_\bk M([n])$ is eventually polynomial.
\end{corollary}

\begin{proof}
Let $\Phi \colon \OI_d \to \FI_d$ be the forgetful functor. Then 
$\rH_M(t) = \rH_{\Phi^*(M)}(t)$, so the statement follows from Theorems~\ref{oithm} and~\ref{hilbthm} except that it only guarantees that $g(t) = \prod_{j=1}^r (1-jt)^{e_j}$ for some $r$. As we saw in the proof of Theorem~\ref{oithm}, for each set $x=[n]$, the lingual structure on $|S_x|$ (in the notation of \S\ref{ss:lingual-cat}) is built on the set $\Sigma = \{0,\dots,d\}$ and $\{1,\dots,d\}$ is the largest repeatable subset, so we can refine this statement using Corollary~\ref{ordhilbcor}.
\end{proof}

\begin{remark} \label{rmk:fi-history}
This result has a history similar to Corollary~\ref{fi-noeth}. It was first proved for $\chr(\bk)=0$ in \cite[Thm.~3.1]{delta-mod}. It was then independently reproved for $d=1$ and $\chr(\bk)=0$ in \cite{fimodules}. Refinements and deeper properties were discovered for these Hilbert series in \cite{symc1} (see for example \cite[\S 6.8]{symc1}).  It was then proved for $d=1$ and any $\bk$ in \cite{fi-noeth}. The result is new if $d>1$ and $\bk$ has positive characteristic.
\end{remark}

\begin{example}
The paper \cite{fimodules} gives many examples of $\FI$-modules appearing ``in nature.'' We now give some examples of $\FI_d$-modules. First a general construction. Let $M_1, \dots, M_d$ be $\FI$-modules over a commutative ring $\bk$. We define $N = M_1 \otimes \cdots \otimes M_d$ to be the following $\FI_d$-module: on sets $S$, it is defined by
\[
N(S) = \bigoplus_{S = S_1 \amalg \cdots \amalg S_d} M_1(S_1) \otimes \cdots \otimes M_d(S_d).
\]
Given a morphism $(f \colon S \to T, g \colon T \setminus f(S) \to [d])$ in $\FI_d$, let $T_i = g^{-1}(i)$ and define $N(S) \to N(T)$ to be the sum of the maps
\[
M_1(S_1) \otimes \cdots \otimes M_d(S_d) \to M_1(S_1 \amalg T_1) \otimes \cdots \otimes M_d(S_d \amalg T_d).
\]
If $M_1, \dots, M_d$ are finitely generated then the same is true for $N$; there is an obvious generalization where we allow each $M_i$ to be a $\FI_{n_i}$-module and $N$ is a $\FI_{n_1 + \cdots + n_d}$-module.

Suppose now, for simplicity, that $\bk$ is a field. For a topological space $Y$ and finite set $S$, let $\Conf_S(Y)$ be the space of injective functions $S \to Y$. For any $r \ge 0$, we get an $\FI$-module $S \mapsto \rH^r(\Conf_S(Y); \bk)$, which we denote by $\rH^r(\Conf(Y))$. Let $X_1, \dots, X_d$ be connected topological spaces, and let $X$ be their disjoint union. The K\"unneth theorem gives
\begin{displaymath}
\rH^r(\Conf_S(X); \bk) = \bigoplus_{p_1+\cdots+p_d = r} \bigoplus_{S = S_1 \amalg \cdots \amalg S_d} \rH^{p_1}(\Conf_{S_1}(X_1); \bk) \otimes \cdots \otimes \rH^{p_d}(\Conf_{S_d}(X_d); \bk).
\end{displaymath}
In particular, we can write
\[
\rH^r(\Conf(X)) = \bigoplus_{p_1+ \cdots + p_d = r} \rH^{p_1}(\Conf(X_1)) \otimes \cdots \otimes \rH^{p_d}(\Conf(X_d)),
\]
so that we can endow $\rH^r(\Conf(X))$ with the structure of an $\FI_d$-module. Under reasonable hypotheses on the $X_i$, each $\rH^r(\Conf(X_i))$ is a finitely generated $\FI$-module (see \cite[Theorem E]{fi-noeth}), and so $\rH^r(\Conf(X))$ will be finitely generated as an $\FI_d$-module.
\end{example}

\subsection{Twisted commutative algebras}
\label{ss:tca}

We now discuss the relationship between $\FI_d$ and twisted commutative algebras. For this part, we assume $\bk$ is commutative. 

\begin{definition}
A {\bf twisted commutative algebra} (tca) over $\bk$ is an associative and unital graded $\bk$-algebra $A=\bigoplus_{n \ge 0} A_n$ equipped with an action of $S_n$ on $A_n$ such that:
\begin{enumerate}
\item the multiplication map $A_n \otimes A_m \to A_{n+m}$ is $(S_n \times S_m)$-equivariant (where we use the standard embedding $S_n \times S_m \subset S_{n+m}$ for the action on $A_{n+m}$); and 
\item given $x \in A_n$ and $y \in A_m$ we have $xy=(yx)^{\tau}$, where $\tau=\tau_{m,n} \in S_{n+m}$ is defined by
\begin{displaymath}
\tau(i) = \begin{cases} i+n & \text{if } 1 \le i \le m, \\ i-m & \text{if } m+1 \le i \le n+m. \end{cases}
\end{displaymath}
This is the ``twisted commutativity'' condition. 
\end{enumerate}

An {\bf $A$-module} is a graded $A$-module $M=\bigoplus_{n \ge 0} M_n$ (in the usual sense) equipped with an action of $S_n$ on $M_n$ such that multiplication $A_n \otimes M_m \to M_{n+m}$ is $(S_n \times S_m)$-equivariant.
\end{definition}

\begin{example}
Let $x_1, \ldots, x_d$ be indeterminates, each regarded as having degree~1. We define a tca $A=\bk\langle x_1, \ldots, x_d \rangle$, the {\bf polynomial tca} in the $x_i$, as follows. As a graded $\bk$-algebra, $A$ is just the non-commutative polynomial ring in the $x_i$. The $S_n$-action on $A_n$ is the obvious one: on monomials it is given by $\sigma(x_{i_1} \cdots x_{i_n})=x_{i_{\sigma(1)}} \cdots x_{i_{\sigma(n)}}$.
\end{example}

We now give a more abstract way to define tca's, which clarifies some constructions. Define a representation of $S_{\ast}$ over $\bk$ to be a sequence $M=(M_n)_{n \ge 0}$, where $M_n$ is a representation of $S_n$ over $\bk$. Given $S_{\ast}$-representations $M$ and $N$, we define an $S_{\ast}$-representation $M \otimes N$ by
\begin{displaymath}
(M \otimes N)_n=\bigoplus_{i+j=n} \Ind_{S_i \times S_j}^{S_n} (M_i \otimes_{\bk} N_j).
\end{displaymath}
There is a symmetry of the tensor product obtained by switching the order of $M$ and $N$ and conjugating $S_i \times S_j$ to $S_j \times S_i$ in $S_n$ via $\tau_{i,j}$. This gives the category $\Rep_{\bk}(S_{\ast})$ of $S_{\ast}$-representations the structure of a symmetric abelian tensor category. The general notions of commutative algebra and module in such a category specialize to tca's and their modules.

\begin{example}
For an integer $n \ge 0$ let $\bk \langle n \rangle$ be the $S_{\ast}$-representation which is the regular representation in degree $n$ and 0 in all other degrees. The tca $\bk\langle x_1, \ldots, x_d \rangle$ can then be identified with the symmetric algebra on the object $\bk\langle 1 \rangle^{\oplus d}$. The symmetric algebra on $\bk\langle n \rangle$ is poorly understood for $n>1$ (although some results are known for $n=2$).
\end{example}

\begin{remark}
If $\bk$ is a field of characteristic $0$ then $\Rep_{\bk}(S_{\ast})$ is equivalent to the category of polynomial representations of $\GL(\infty)$ over $\bk$, as symmetric abelian tensor categories. Under this equivalence, tca's correspond to commutative associative unital $\bk$-algebras equipped with a polynomial action of $\GL(\infty)$. For example, the polynomial tca $\bk \langle x \rangle$ corresponds to the polynomial ring $\bk[x_1, x_2, \ldots]$ in infinitely many variables, equipped with its usual action of $\GL(\infty)$. This point of view has been exploited by us \cite{delta-mod,expos,symc1,infrank} to establish properties of tca's in characteristic $0$.
\end{remark}

\begin{proposition} \label{FI-tca}
Let $A=\bk\langle x_1, \ldots, x_d \rangle$. Then $\Mod_A$ is equivalent to $\Rep_{\bk}(\FI_d)$.
\end{proposition}

\begin{proof}
Pick $M \in \Rep_\bk(\FI_d)$ and let $M_m = M([m])$. Then $M_m$ has an action of $S_m$. To define $A_n \otimes M_m \to M_{n+m}$, it is enough to define how $x_{i_1} \cdots x_{i_n}$ acts by multiplication for each $(i_1, \dots, i_n)$. Let $(f,g) \colon [m] \to [n+m]$ be the morphism in $\FI_d$ where $f \colon [m] \to [n+m]$ is the injection $i \mapsto n+i$ and $g \colon [n] \to [d]$ is the function $g(j) = i_j$ and let the multiplication map be the induced map $M_{(f,g)} \colon M_m \to M_{n+m}$. By definition this is $(S_n \times S_m)$-equivariant and associativity follows from associativity of composition of morphisms in $\FI_d$.

It is easy to reverse this process: given an $A$-module $M$, we get a functor defined on the full subcategory of $\FI_d$ on objects of the form $[n]$ (note that every morphism $[m] \to [n+m]$ is a composition of the injections we defined above with an automorphism of $[n+m]$). To extend this to all of $\FI_d$, pick a total ordering on each finite set to identify it with $[n]$. The two functors we have defined are quasi-inverse to each other.
\end{proof}

\begin{remark}
In particular, the category of $\FI$-modules studied in \cite{fimodules, fi-noeth} is equivalent to the category of modules over the univariate tca $\bk\langle x \rangle$.
\end{remark}

There is an obvious notion of generation for a set of elements in a tca or a module over a tca. We say that a tca is {\bf noetherian} if every subrepresentation of a finitely generated module is again finitely generated.

\begin{corollary}
A tca over a noetherian ring $\bk$ finitely generated in degree $1$ is noetherian.
\end{corollary}

\begin{proof}
A tca finitely generated in degree $1$ is a quotient of $\bk\langle x_1, \ldots, x_d \rangle$ for some $d$.
\end{proof}

\begin{remark}
The same techniques also prove that a twisted graded-commutative algebra finitely generated in degree $1$ is noetherian.
\end{remark}

Suppose $\bk$ is a vector space. We define the {\bf Hilbert series} of a module $M$ over a tca by
\begin{displaymath}
\rH_M'(t) = \sum_{n \ge 0} \dim_{\bk}M([n]) \frac{t^n}{n!}.
\end{displaymath}
As the equivalence in Proposition~\ref{FI-tca} is clearly compatible with Hilbert series, we obtain:

\begin{corollary}
Let $M$ be a finitely generated $\bk \langle x_1, \ldots, x_d \rangle$-module. Then $\rH'_M(t) \in \bQ[t,e^t]$.
\end{corollary}

\begin{remark}
Twisted Lie algebras, which are related structures, appear in \cite{barratt}, and twisted commutative algebras appear in \cite{ginzburg-schedler}.
\end{remark}

\subsection{Additional results} \label{sec:FI-additional}

\begin{proposition}
The category $\FI$ satisfies property~{\rm (F)} {\rm (}see Definition~\ref{def:catF}{\rm )}.
\end{proposition}

\begin{proof}
Let $x$ and $x'$ be sets, and consider maps $f \colon x \to y$ and $f' \colon x' \to y$. Let $\ol{y}$ be the union of the images of $f$ and $f'$. Let $g \colon \ol{y} \to y$ be the inclusion, and write $\ol{f}$ for the map $x \to \ol{y}$ induced by $f$, and similarly for $\ol{f}'$. Then $f=g \circ \ol{f}$ and $f'=g \circ \ol{f}'$. Since $\# \ol{y} \le \# x + \# x'$, it follows that there are only finitely many choices for $(\ol{y}, \ol{f}, \ol{f}')$, up to isomorphism. This completes the proof.
\end{proof}

The next result follows from Proposition~\ref{pointwise} and reproves \cite[Proposition 2.61]{fimodules}.

\begin{corollary}
For any commutative ring $\bk$, the pointwise tensor product of finitely generated representations of $\FI$ is finitely generated.
\end{corollary}

\begin{remark}
The above corollary is false for $\FI_d$ for $d>1$. It follows that these categories do not satisfy property~(F). To see this, consider $M = P_1 \odot P_1$. Then $\dim_\bk M([n]) = (d^2)^n$, so the Hilbert series is $\rH_M(t) = (1-d^2 n)^{-1}$. By Corollary~\ref{fihilb}, the Hilbert series of a finitely generated $\FI_d$-module cannot have this form if $d>1$. 
\end{remark}

The next results concern the category $\FA$ of finite sets, with all functions as morphisms.

\begin{theorem}  \label{thm:set-noeth}
The category $\FA$ is quasi-Gr\"obner.
\end{theorem}

\begin{proof}
We claim that the inclusion functor $\Phi \colon \FI \to \FA$ satisfies property~(F). Let $x$ be a set and let $f_i \colon x \to y_i$ be representatives for the isomorphism classes of surjective maps from $x$; these are finite in number. Any morphism $f \colon x \to y$ factors as $f=g \circ f_i$ for some $i$ and some injective map $g \colon y_i \to y$ which proves the claim. Since $\FI$ is quasi-Gr\"obner, the result now follows from Proposition~\ref{potgrob}.
\end{proof}

\begin{corollary} \label{cor:set-noeth}
If $\bk$ is left-noetherian then $\Rep_{\bk}(\FA)$ is noetherian.
\end{corollary}

Suppose $\bk$ is a field. By what we have shown, the Hilbert series of a finitely generated $\FA$-module is of the form $f(t) / (1-t)^d$ for some polynomial $f$ and $d \ge 0$. Equivalently, the function $n \mapsto \dim_\bk M([n])$ is eventually polynomial. In fact, one can do better:

\begin{theorem} \label{thm:FA-poly}
If $\bk$ is a field and $M$ is a finitely generated $\FA$-module, then the function $n \mapsto \dim_\bk M([n])$ agrees with a polynomial for all $n > 0$. Equivalently, the Hilbert series of $M$ is of the form $f(t) / (1-t)^d$ where $\deg f \le d$.
\end{theorem}

\begin{proof}
Let $\FA^\circ$ be the category of nonempty finite sets, so that every $\FA$-module gives an $\FA^\circ$-module by restriction (this notation will only be used in this proof). Introduce an operator $\Sigma$ on $\FA^\circ$-modules by $(\Sigma M)(S) = M(S \amalg \{*\})$. It is easy to see that $\Sigma M$ is a finitely generated $\FA^\circ$-module if the same is true for $M$. For every nonempty set $S$, the inclusion $S \to S \amalg \{*\}$ can be split, and so the map $M(S) \to (\Sigma M)(S)$ is an inclusion for all $S$. Define $\Delta M = \Sigma M / M$. Since $\Sigma$ is exact, it follows that $\Delta M \to \Delta M'$ is a surjection if $M \to M'$ is a surjection.

Define $h_M(n) = \dim_\bk M([n])$. Recall that a function $f \colon \bZ_{> 0} \to \bZ$ is a polynomial of degree $\le d$ if and only if the function $g(n) = f(n+1)-f(n)$ is a polynomial of degree $\le d-1$. Since $h_{\Delta M}(n) = h_M(n+1) - h_M(n)$, we get that $h_M$ is a polynomial function of degree $\le d$ if and only if $\Delta^{d+1} M = 0$. Pick a surjection $P \to M \to 0$ with $P$ a direct sum of principal projectives $P_S$. Since $h_{P_S}(n) = n^{|S|}$ is a polynomial, $P$ is annihilated by some power of $\Delta$. Since $\Delta$ preserves surjections, the same is true for $M$, and so $h_M(n)$ is a polynomial.
\end{proof}

\begin{remark}
\begin{enumerate}
\item The category of representations $\Rep_{\bk}(\FA)$ is studied in \cite{wiltshire} in the case that $\bk$ is a field of characteristic $0$, and Theorems~\ref{thm:FA-poly} and \ref{thm:FA-artinian} are proved.

\item In \cite{dougherty}, the functions $n \mapsto |F([n])|$ coming from functors $F \colon \FA \to \FA$ are classified.

\item The analogue of \cite[Proposition 4.10]{kuhnI} for $\FA$ holds with the same proof, i.e., if $\deg h_M(n) = r$, then the lattice of $\FA$-subrepresentations of $M$ is isomorphic to the lattice of $\bk[\End([r])]$-subrepresentations of $M([r])$. Using Theorem~\ref{thm:FA-poly}, this implies that finitely generated $\FA$-modules are finite length. We will give a different proof using Gr\"obner methods in Theorem~\ref{thm:FA-artinian}. \qedhere
\end{enumerate}
\end{remark}

\section{Categories of surjections} \label{sec:cat-surj}

In this section, we study $\FS^{\op}$, the (opposite of the) category of finite sets with surjective functions and variations, which behaves quite differently from the category of injective functions. The main results are stated in \S\ref{sec:OSFS-defn} and the proofs are in \S\ref{sec:proof-osthm}. The results on $\FS^{\op}$ are powerful enough to prove the Lannes--Schwartz artinian conjecture; this is done in \S\ref{sec:linear-cat}. We give some complementary results in \S\ref{sec:surj-additional}.

\subsection{The categories $\OS$ and $\FS$} \label{sec:OSFS-defn}

Define $\FS$ to be the category whose objects are non-empty finite sets and whose morphisms are surjective functions. We define an ordered version $\OS$ of this category as follows. The objects are totally ordered finite sets. A morphism $S \to T$ in $\OS$ is an {\bf ordered surjection}: a surjective map $f \colon S \to T$ such that for all $i<j$ in $T$ we have $\min{f^{-1}(i)} < \min{f^{-1}(j)}$. We norm $\FS$ and $\OS$ over $\bN$ by $\nu(x)=\# x$.

Our main result about $\OS$ is the following theorem, which is proved in the next section.

\begin{theorem} \label{osthm}
The category $\OS^\op$ is $\rO$-lingual and Gr\"obner.
\end{theorem}

\begin{theorem} \label{thm:FSop-PG}
The forgetful functor $\Phi \colon \OS^\op \to \FS^\op$ satisfies property~{\rm (F)}. In particular, the category $\FS^\op$ is quasi-Gr\"obner.
\end{theorem}

\begin{proof}
Similar to the proof of Theorem~\ref{fithm}.
\end{proof}

\begin{corollary} \label{fs-noeth}
If $\bk$ is left-noetherian then $\Rep_{\bk}(\FS^{\op})$ is noetherian.
\end{corollary}

\begin{corollary} \label{cor:FSop-hilbert}
Let $M$ be a finitely generated $\FS^\op$-module defined over a field $\bk$. Then the Hilbert series
\begin{displaymath}
\rH_M(t) = \sum_{n \ge 0} \dim_\bk{M([n])} \cdot t^n
\end{displaymath}
has the form $f(t)/g(t)$, where $f(t)$ and $g(t)$ are polynomials, and $g(t)$ factors as $\prod_{j=1}^r (1-jt)^{e_j}$ for some $r$ and $e_j \ge 0$. If $M$ is generated in degree $\le d$, then we can take $r=d$.
\end{corollary}

\begin{proof}
If $\Phi \colon \OS^\op \to \FS^\op$ is the forgetful functor, then $\rH_M(t) = \rH_{\Phi^*(M)}(t)$, so the result follows from Theorems~\ref{osthm} and~\ref{hilbthm}. To prove the last statement, we note that it follows from \S\ref{sec:proof-osthm} that for a finite set $x=[n]$, the lingual structure on $|S_x|$ (in the notation of \S\ref{ss:lingual-cat}) is built on the set $\{1,\dots,n\}$. Now use Corollary~\ref{ordhilbcor}.
\end{proof}

\begin{remark}
Using partial fraction decomposition, a function $f(t) / (\prod_{j=1}^r (1-jt)^{e_j})$ can be written as a sum $\sum_{j=1}^r f_j(t) / (1-jt)^{e_j}$ for polynomials $f_j(t) \in \bQ[t]$. Corollary~\ref{cor:FSop-hilbert} implies that if $M$ is a finitely generated $\FS^\op$-module over a field $\bk$, then there exist polynomials $p_1, \dots, p_r$ so that the function $n \mapsto \dim_\bk M([n])$ agrees with $\sum_{j=1}^r p_j(n) j^n$ for $n \gg 0$.
\end{remark}

The following more general form of the corollary will be used in \S\ref{sec:proof-delta-new} for the proof of Theorem~\ref{thm:delta-new}. 

\begin{corollary} \label{FShilb2}
Let $M$ be a finitely generated $(\FS^{\op})^r$-representation over a field $\bk$. Then the Hilbert series
\begin{displaymath}
\rH_M(t) = \sum_{\bn \in \bN^r} \dim_\bk{M([\bn])} \cdot \bt^\bn
\end{displaymath}
is a rational function $f(\bt)/g(\bt)$, where $g(\bt)$ factors as $\prod_{i=1}^r \prod_{j=1}^{\infty} (1-jt_i)^{e_{i,j}}$, where all but finitely many $e_{i,j}$ are $0$. Equivalently, the exponential Hilbert series
\begin{displaymath}
\rH'_M(t) = \sum_{\bn \in \bN^r} \dim_\bk{M([\bn])} \cdot \frac{\bt^\bn}{\bn!},
\end{displaymath}
is a polynomial in the $t_i$ and the $e^{t_i}$. {\rm (}Here $\bn!=n_1! \cdots n_r!$.{\rm )}
\end{corollary}

\begin{proof}
Proposition~\ref{prodling} shows that $(\OS^{\op})^r$, given the product norm over $I=\{1, \ldots, r\}$, is O-lingual. However, to get the stated result we need to extract slightly more information from the proofs. The languages appearing in the proof of Proposition~\ref{prodling} all have the form $\cL_1 \cdots \cL_r$, where $\cL_i$ is an ordered language normed over $\bN^{\{i\}} \subset \bN^I$. It follows that the repeatable subsets of $\cL$ are singletons. Thus, when applying Theorem~\ref{ordhilb} to the relevant languages, each $\lambda_i$ is a non-negative multiple of a single $t_j$.
\end{proof}

\begin{remark}
The proof of Corollary~\ref{FShilb2} can be seen in the wider context of a theory of O-lingual normed categories in which one keeps track of a partition of the set $I$ that is compatible with the norm in a certain sense.
\end{remark}

\subsection{Proof of Theorem~\ref{osthm}} \label{sec:proof-osthm}

Let $\Sigma$ be a finite set. Let $s \colon [n] \to \Sigma$ and $t \colon [m] \to \Sigma$ be two elements of $\Sigma^{\star}$. We define $s \le t$ if there exists an order-preserving injection $\phi \colon [n] \to [m]$ such that for all $i \in [n]$ we have $s_i=t_{\phi(i)}$, and for all $j \in [m]$ there exists $j' \le j$ in the image of $\phi$ with $t_j=t_{j'}$. This defines a partial order on $\Sigma^{\star}$.

\begin{proposition} \label{prop:secondorder-noeth}
The poset $\Sigma^{\star}$ is noetherian and every ideal is an ordered language.
\end{proposition}

\begin{proof}
Suppose that $\Sigma^{\star}$ is not noetherian. We use the notion of minimal bad sequences from the proof of Theorem~\ref{thm:higman}. Let $x_1, x_2, \dots$ be a minimal bad sequence in $\Sigma^\star$.

Given $x \in \Sigma^{\star}$ call a letter of $x$ {\it exceptional} if it appears exactly once. If $x$ has a non-exceptional letter, then let $m(x)$ denote the index, counting from the end, of the first non-exceptional letter. Note that $m(x) \le \# \Sigma$. Also, $\ell(x) > \# \Sigma$ implies that $x$ has some non-exceptional letter, so $m(x_i)$ is defined for $i \gg 0$. (Note: $\ell(x_i) \to \infty$, as there are only finitely many sequences of a given length.) 

So we can find an infinite subsequence of $x_1, x_2, \dots$ on which $m$ is defined and is constant, say equal to $m_0$. We can find a further infinite subsequence where the letter in the position $m_0$ is constant. Call this subsequence $x_{i_1}, x_{i_2}, \dots$ and let $y_{i_j}$ be the sequence $x_{i_j}$ with the $m_0$th position (counted from the end) deleted. By construction, the sequence $x_1, \dots, x_{i_1-1}, y_{i_1}, y_{i_2}, \dots$ is not bad and so some pair is comparable. Note that $y_{i_j} \le x_{i_j}$, so $x_i$ and $y_{i_j}$ are incomparable, and so we have $j<k$ such that $y_{i_j} \le y_{i_k}$. But this implies $x_{i_j} \le x_{i_k}$, contradicting that $x_1, x_2, \ldots$ forms a bad sequence. Thus $\Sigma^{\star}$ is noetherian.

We now show that a poset ideal of $\Sigma^{\star}$ is an ordered language. It suffices to treat the case of principal ideals; the principal ideal generated by $w = w_1 \cdots w_n$ is the language
\begin{displaymath}
w_1 \Pi_1^{\star} w_2 \Pi_2^{\star} \cdots w_n \Pi_n^{\star},
\end{displaymath}
where $\Pi_i=\{w_1, \ldots, w_i\}$, which is clearly ordered.
\end{proof}

\begin{proof}[Proof of Theorem~\ref{osthm}]
Clearly, $\cC=\OS^{\op}$ is directed. Let $n$ be a non-negative integer, and regard $x=[n]$ as an object of $\cC$. A morphism $f \colon [m] \to [n]$ can be regarded as a word of length $m$ in the alphabet $\Sigma=[n]$. In this way, we have an injective map $i \colon \vert \cC_x \vert \to \Sigma^{\star}$. This map is strictly order-preserving with respect to the order on $\Sigma^{\star}$ defined above, and so Proposition~\ref{prop:secondorder-noeth} implies that $\vert \cC_x \vert$ is noetherian. The lexicographic order on words (with respect to any total order on $\Sigma$) induces an admissible order on $\vert \cC_x \vert$. For $f \in \vert \cC_x \vert$ we have $m=\nu(f)=\ell(i(f))$, and so we have a lingual structure on $\vert \cC_x \vert$ if we norm words in $\Sigma^\star$ by their length. Finally, an ideal of $\vert \cC_x \vert$ gives an ordered language over $\Sigma$ by Proposition~\ref{prop:secondorder-noeth}, and so this is an O-lingual structure.
\end{proof}

\subsection{Linear categories} \label{sec:linear-cat}

Let $R$ be a finite commutative ring. A linear map between free $R$-modules is {\bf splittable} if the image is a direct summand. Let $\VA_R$ (resp.\ $\VI_R$) be the category whose objects are finite rank free $R$-modules and whose morphisms are splittable maps (resp.\ splittable injections).

\begin{theorem} \label{thm:VI-noeth}
The categories $\VI_R$ and $\VA_R$ are quasi-Gr\"obner.
\end{theorem}

\begin{proof}
Define a functor $\Phi \colon \FS^\op \to \VI_R$ by $S \mapsto \hom_R(R[S], R) = R[S]^*$. It is clear that $\Phi$ is essentially surjective. We claim that $\Phi$ satisfies property~(F). Fix $U \in \VI_R$. Pick a finite set $S$ and a splittable injection $f \colon U \to R[S]^*$. Dualize this to get a surjection $R[S] \to U^*$. Letting $T \subseteq U^*$ be the image of $S$ under this map, the map factorizes as $R[S] \to R[T] \to U^*$ where the first map comes from a surjective function $S \to T$. So we can take $y_1, y_2, \dots$ to be the set of subsets of $U^*$ which span $U^*$ as an $R$-module (there are finitely many of them since $R$ is finite) and $f_i \colon U \to R[T]^*$ to be the dual of the natural map $R[T] \to U^*$. This establishes the claim. Also, as in the proof of Theorem~\ref{thm:set-noeth}, the inclusion functor $\VI_R \to \VA_R$ satisfies property~(F). Since $\FS^\op$ is quasi-Gr\"obner (Theorem~\ref{thm:FSop-PG}), the theorem follows from Proposition~\ref{potgrob}.
\end{proof}

\begin{remark}
The idea of using the functor $\Phi$ in the above proof was communicated to us by Aur\'elien Djament after a first version of these results was circulated. The original version of the above proof involved working with a version of $\VI_R$ consisting of spaces with ordered bases and upper-triangular linear maps. This idea is no longer needed to prove the desired properties for $\VI_R$, but is still needed in \cite{putman-sam} to prove noetherianity of a related category $\mathbf{VIC}_R$ (splittable injections plus a choice of splitting). 
\end{remark}

\begin{corollary} \label{VI-noeth}
If $\bk$ is left-noetherian then $\Rep_\bk(\VI_R)$ and $\Rep_\bk(\VA_R)$ are noetherian.
\end{corollary}

\begin{corollary}
Let $M$ be a finitely generated $\VI_R$-module or $\VA_R$-module over a field $\bk$. Then the Hilbert series
\begin{displaymath}
\rH_M(t) = \sum_{n \ge 0} \dim{M_\bk([n])} \cdot t^n
\end{displaymath}
is a rational function of the form $f(t)/g(t)$ where $g(t)$ factors as $\prod_{j=1}^r (1-jt)^{e_r}$, for some $r$.
\end{corollary}
 
\begin{remark}
When $R = \bk = \bF_q$ is a finite field, Corollary~\ref{VI-noeth} proves the Lannes--Schwartz artinian conjecture \cite[Conjecture 3.12]{kuhnII}. This is also a consequence of the results in \cite{putman-sam}, and so a similar, but distinct, proof of the artinian conjecture appears there as well. The ideas for proving this conjecture were only made possible by work from both projects.
\end{remark}

\begin{remark}
While preparing this article, we learned that \cite{ganli} proves that $\Rep_\bk(\VI_R)$ is noetherian in the special case that $R$ is a field and $\bk$ is a field of characteristic $0$.
\end{remark}

\subsection{Additional results} \label{sec:surj-additional}

Recall that $\FA$ is the category of finite sets.

\begin{theorem} \label{thm:FAop-qg}
The category $\FA^\op$ is quasi-Gr\"obner.
\end{theorem}

\begin{proof}
The proof is similar to that of Theorem~\ref{thm:set-noeth}; the role of $\FI$ is played by $\FS^{\op}$.
\end{proof}

\begin{corollary} \label{setop-noeth}
If $\bk$ is left-noetherian then $\Rep_\bk(\FA^{\op})$ is noetherian.
\end{corollary}

\begin{proposition} \label{FA:propD}
The category $\FA$ satisfies property~{\rm (D)} {\rm (}see Definition~\ref{defn:propD}{\rm )}.
\end{proposition}

\begin{proof}
Let $x$ be a given finite set. Let $y=2^x$ be the power set of $x$, and let $S \subset \Hom(x,y)$ be the set of maps $f$ for which $a \in f(a)$ for all $a \in x$. Let $f \colon x \to z$ be given. Define $g \colon z \to y$ by $g(a)=f^{-1}(a)$. If $f' \colon x \to z$ is an arbitrary map then $g(f'(a))=f^{-1}(f'(a))$ contains $a$ if and only if $f(a)=f'(a)$. Thus $g_*(f')$ belongs to $S$ if and only if $f=f'$. We conclude that $g_*^{-1}(S)=\{f\}$, which establishes the proposition.
\end{proof}

We can now give our improvement of Corollary~\ref{cor:set-noeth}:

\begin{theorem} \label{thm:FA-artinian}
If $\bk$ is a field, then every finitely generated $\FA$-module is artinian, and so has a finite composition series. In other words, $\Rep_\bk(\FA)$ has Krull dimension $0$.
\end{theorem}

\begin{proof}
We can use Proposition~\ref{dim0} since $\FA$ is quasi-Gr\"obner (Theorem~\ref{thm:set-noeth}), $\FA^{\op}$ is quasi-Gr\"obner (Theorem~\ref{thm:FAop-qg}), and $\FA$ satisfies property~(D) (Proposition~\ref{FA:propD}).
\end{proof}

\begin{remark} \label{rmk:VA-not-artinian}
Since $\VA_{\bF_q}$ is like a linear version of $\FA$, it is tempting to seek a linear analogue of Theorem~\ref{thm:FA-artinian}. When $q$ is invertible in $\bk$, such an analog is the main result of \cite{kuhn:mixed}. However, when $\bk=\bF_q$, the analogous statement is false: the Krull dimension of $P_{\bF_q}$ is~1, as can be read off from the explicit description of the subrepresentation lattice of $P_{\bF_q}^\vee$ given in \cite[Theorem 6.4]{kuhn:survey}. It is conjectured that the Krull dimension of $P_W$ is $\dim(W)$ in general \cite[Conjecture 6.8]{kuhn:survey}.
\end{remark}

\begin{remark} \label{rmk:FAop-not-artinian}
$\Rep_{\bk}(\FA^{\op})$ is not artinian: the principal projective $P_S$ in $\FA$ grows like a polynomial of degree $|S|$, i.e., $|\hom_\FA(S,T)| = |T|^{|S|}$, but the principal projectives in $\FA^\op$ grow like exponential functions, so the duals of projective objects in $\FA^\op$ cannot be finitely generated $\FA$-modules. In particular, $\FA^{\op}$ does not satisfy property~(D).
\end{remark}

\begin{remark}
In \cite[Theorem 1.3]{dougherty} and \cite[Corollary 4.4]{pare}, the functions $n \mapsto |F([n])|$ coming from functors $F \colon \FA^\op \to \FA$ are classified.
\end{remark}

\begin{proposition}
The category $\FS^{\op}$ satisfies property~{\rm (F)}.
\end{proposition}

\begin{proof}
Let $x_1$ and $x_2$ be given finite sets. Given surjections $f_1 \colon y \to x_1$ and $f_2 \colon y \to x_2$, let $y'$ be the image of $y$ in $x_1 \times x_2$, and let $g \colon y \to y'$ be the quotient map. Then $f_i = p_i \circ g$, where $p_i \colon x_1 \times x_2 \to x_i$ is the projection map. Since there are only finitely many choices for $(y', p_1, p_2)$ up to isomorphism, the result follows. 
\end{proof}

\begin{corollary} \label{FSpointwise}
For any commutative ring $\bk$, the pointwise tensor product of finitely generated representations of $\FS^{\op}$ is finitely generated.
\end{corollary}

\begin{proposition} \label{FAhomrep}
Let $S$ be a finite set and let $M$ be the $\FS^{\op}$-module defined by $M(T)=\bk[\Hom_{\FA}(T, S)]$. Then $M$ is finitely generated, and hence noetherian.
\end{proposition}

\begin{proof}
Let $S_1, \ldots, S_n$ be the subsets of $S$. Then $M(T)=\bigoplus_{i=1}^n \bk[\Hom_{\FS}(T, S_i)]$. The $i$th summand is exactly the principal projective at $S_i$.
\end{proof}

\section{\texorpdfstring{$\Delta$}{Delta}-modules}
\label{sec:Delta-mod}

We now apply our methods to the theory of $\Delta$-modules, originally introduced by the second author in \cite{delta-mod}. In \S \ref{ss:deltabg}, we recall the definitions, the results of \cite{delta-mod}, and state our new results. In \S \ref{delta-example}, we explain how our results yield new information on syzygies of Segre embeddings. The remainder of \S \ref{sec:Delta-mod} is devoted to the proofs.

Throughout, $\bk$ is a field and $\Vec$ is the category of finite-dimensional $\bk$-vector spaces. We write $\Vec^n$ for the $n$-fold direct product of $\Vec$ with itself.

\subsection{Polynomial functors} \label{ss:polyfunc}

We review some background on polynomial functors. This material is standard; for example, see \cite[\S I.A]{macdonald}, \cite{friedlander-suslin}. See \cite[\S 1]{FFSS} for a comparison with the functors of finite degree in \cite{eilenberg-maclane}, which is a weaker notion than what we need.

A {\bf (strict) polynomial functor} is a functor $F \colon \Vec \to \Vec$ such that for all $V$ and $W$, the map
\[
\hom_\bk(V,W) \to \hom_\bk(F(V), F(W))
\]
is given by polynomial functions (these polynomials are unique if $\bk$ is infinite, but otherwise are part of the data) such that the degrees of the functions are bounded as we vary over all $V,W$ (the bounded condition is usually called ``finite degree'' but we only deal with such functors). A polynomial functor is {\bf homogeneous} of degree $d$ if the map above is given by degree $d$ homogeneous polynomials. Every polynomial functor is a finite direct sum of homogeneous functors \cite[Lemma 2.5]{friedlander-suslin}. 

It will be convenient to have an alternative description of homogeneous polynomial functors. Given an integer $d \ge 0$, let $\rD^d(\Vec)$ be the category whose objects are finite-dimensional $\bk$-vector spaces and 
\[
\hom_{\rD^d(\Vec)}(U,U') = \rD^d\hom_\bk(U,U').
\]
(Recall that D is the divided power functor.) Composition is defined using the natural maps
\begin{align*}
\rD^d(\hom_\bk(U, U')) \otimes \rD^d(\hom_\bk(U', U'')) 
&\to \rD^d(\hom_\bk(U,U') \otimes \hom_\bk(U', U''))\\
& \to \rD^d(\hom_\bk(U,U'')).
\end{align*}
A homogeneous polynomial functor of degree $d$ is the same as a linear functor $\rD^d(\Vec) \to \Vec$.

We extend this definition to multivariate functors: a {\bf (strict) polynomial functor} is a functor $F \colon \Vec^n \to \Vec$ such that for all $(V_1,\dots,V_n)$ and $(W_1,\dots,W_n)$, the map
\[
\hom_\bk(V_1,W_1) \times \cdots \times \hom_\bk(V_n,W_n) \to \hom_\bk(F(V_1,\dots,V_n), F(W_1,\dots,W_n))
\]
is given by polynomial functions (again, these functions are unique if $\bk$ is infinite, but otherwise are part of the data) whose degrees are bounded as we vary over $V_i,W_i$. A polynomial functor is homogeneous of multidegree $\bd = (d_1,\dots,d_n)$ if these functions are homogeneous polynomials of multidegree $\bd$. Again, every polynomial functor is a finite direct sum of homogeneous functors. A homogeneous polynomial functor of multidegree $\bd$ is the same as a linear functor $\rD^{\bd}(\Vec) \to \Vec$, where $\rD^{\bd}(\Vec)$ is the category whose objects are $n$-tuples of vector spaces $\ul{V} = (V_1, \dots, V_n)$ and whose morphisms are 
\[
\hom_{\rD^{\bd}(\Vec)}(\ul{V}, \ul{W}) = \rD^{d_1}(\hom_\bk(V_1, W_1)) \otimes \cdots \otimes \rD^{d_n}(\hom_\bk(V_n, W_n)).
\]

Given vector spaces $\ul{U}=(U_1, \dots, U_n)$, define a homogeneous polynomial functor $P_{\ul{U}}$ by 
\[
P_{\ul{U}}(\ul{V}) = \hom_{\rD^{\bd}(\Vec)}(\ul{U}, \ul{V}).
\]
For a partition $\lambda = (\lambda_1, \dots, \lambda_s)$, set $\rD^\lambda(V) = \rD^{\lambda_1}(V) \otimes \cdots \otimes \rD^{\lambda_s}(V)$.  For a sequence of partitions $\Lambda = (\lambda^1, \dots, \lambda^n)$, define $\rD^\Lambda \colon \Vec^n \to \Vec$ by
\begin{displaymath}
\rD^\Lambda(V_1, \ldots, V_r) = \rD^{\lambda^1}(V_1) \otimes \cdots \otimes \rD^{\lambda^n}(V_n).
\end{displaymath}

\begin{proposition} \label{prop:Dlambda-projgen}
The $\rD^{\Lambda}$ are projective generators for the category of polynomial functors $\Vec^n \to \Vec$.
\end{proposition}

\begin{proof}
Yoneda's lemma implies that the $P_{\ul{U}}$ are projective, so it remains to show that they generate. For $n=1$, this statement is \cite[Theorem 2.10]{friedlander-suslin}.

For general $n$, let $F$ be a polynomial functor. Given vector spaces $V_1, \dots, V_n$, set $\bV = V_1 \oplus \cdots \oplus V_n$ and define $G \colon \Vec^n \to \Vec$ by $G(V_1,\dots,V_n) = F(\bV,\dots,\bV)$. Since the surjections $\bV\to V_i$ split in $\Vec$, we have a surjection $G \to F$. Then $V \mapsto F(V,\dots,V)$ is a polynomial functor in one variable and hence is a quotient of $P$ which is a direct sum of $\rD^{k}$ for various integers $k \ge 0$. So we have a surjection $P(V_1 \oplus \cdots \oplus V_n) \to G(V_1, \dots, V_n)$ which is natural for linear maps $V_i \to V'_i$. Finally, $\rD^k(\bV) = \bigoplus_{\Lambda} \rD^\Lambda(V_1,\dots,V_n)$, where the sum is over all $\Lambda = (\lambda^1,\dots,\lambda^n)$ such that $\sum \lambda^i = d$, so the $\rD^\Lambda$ generate.
\end{proof}

For a polynomial functor $F$, we define $F(\bk^{\infty})$ to be the direct limit of the $F(\bk^n)$, with respect to the standard inclusions $\bk^n \to \bk^{n+1}$. The space $F(\bk^{\infty})$ is a representation of $\GL(\infty)=\bigcup \GL(n)$. Representations arising in this fashion are called polynomial representations. Similar comments apply to $r$-variate polynomial functors, which yield polynomial representations of $\GL(\infty)^r$. 

The action of the subgroup of invertible diagonal matrices on a polynomial representation decomposes as a sum of $1$-dimensional representations corresponding to characters of the subgroup, whose action is of the form $\diag(x_1, x_2, \dots) v = (\prod_i x_i^{\alpha_i}) v$ for some $\alpha_i \in \bZ$. We call $\alpha$ the {\bf weight} of $v$, and also say it is a weight of the ambient representation.

\subsection{Statement of results}
\label{ss:deltabg}

Define a category $\Vec^\Delta$ as follows. The objects are finite collections of vector spaces $\{V_i\}_{i \in I}$. A morphism $\{V_i\}_{i \in I} \to \{W_j\}_{j \in J}$ is a surjection $f \colon J \to I$ together with, for each $i \in S$, a linear map $\eta_i \colon V_i \to \bigotimes_{f(j)=i} W_j$. A {\bf $\Delta$-module} is a polynomial functor $F \colon \Vec^{\Delta} \to \Vec$, where polynomial means that the functor $F_n \colon \Vec^n \to \Vec$ given by $(V_1, \ldots, V_n) \mapsto F(\{V_i\}_{i \in [n]})$ is polynomial for each $n$. We write $\Mod_{\Delta}$ for the category of $\Delta$-modules.

More concretely, a $\Delta$-module is a sequence $(F_n)_{n \ge 0}$, where $F_n$ is an $S_n$-equivariant polynomial functor $\Vec^n \to \Vec$, equipped with transition maps
\begin{equation}
\label{delta-trans}
F_n(V_1, \ldots, V_{n-1}, V_n \otimes V_{n+1}) \to F_{n+1}(V_1, \ldots, V_{n+1})
\end{equation}
satisfying certain compatibilities. 

Let $F$ be a $\Delta$-module. An {\bf element} of $F$ is an element of $F(\{V_i\})$ for some object $\{V_i\}$ of $\Vec^{\Delta}$. Given a set $S$ of elements of $F$, the subrepresentation of $F$ {\bf generated} by $S$ is the smallest $\Delta$-subrepresentation of $F$ containing $S$. Then $F$ is {\bf finitely generated} if it is generated by a finite set of elements, and $F$ is {\bf noetherian} if every $\Delta$-subrepresentation of $F$ is finitely generated.

Let $\Lambda$ be the ring of symmetric functions (see \cite[Ch. 7]{stanley-EC2}). Given a finite length polynomial representation $V$ of $\GL(\infty)^n$ over $\bk$, we define its character $\wt{\ch}(V) \in \Lambda^{\otimes n}$ by
\begin{displaymath}
\wt{\ch}(V) = \sum_{(\lambda_1, \ldots, \lambda_n)} \dim_{\bk}(V_{\lambda_1, \ldots, \lambda_n}) \cdot m_{\lambda_1} \otimes \cdots \otimes m_{\lambda_n},
\end{displaymath}
where the sum is over all $n$-tuples of partitions, $V_{\lambda_1,\ldots,\lambda_n}$ is the $(\lambda_1, \ldots, \lambda_n)$-weight space of $V$, and $m_{\lambda}$ is the monomial symmetric function indexed by $\lambda$. We let $\ch(V)$ be the image of $\wt{\ch}(V)$ in $\Sym^n(\Lambda)$, a polynomial of degree $n$ in the $m_{\lambda}$. If $V$ is an $S_n$-equivariant representation then $\wt{\ch}(V)$ belongs to $(\Lambda^{\otimes n})^{S_n}$, and so passing to $\ch(V)$ does not lose information. If $F \colon \Vec^n \to \Vec$ is a finite length polynomial functor, define $\ch(F):=\ch(F(\bk^{\infty}, \ldots, \bk^{\infty}))$.

Suppose now that $F=(F_n)_{n \ge 0}$ is a $\Delta$-module, and each $F_n$ has finite length. We define two notions of {\bf Hilbert series} of $F$:
\begin{displaymath}
\rH_F'(\bm) = \sum_{n \ge 0} \frac{\ch(F_n)}{n!}, \qquad \rH_F(\bm) = \sum_{n \ge 0} \ch(F_n).
\end{displaymath}
These are elements of $\SYM(\Lambda_{\bQ}) \cong \bQ \lbb \bm \rbb$. If $F$ is finitely generated, then each $F_n$ is finite length; moreover, only finitely many of the $m_{\lambda}$ appear in $\rH'_F(\bm)$ and $\rH_F(\bm)$.

In \cite{delta-mod}, a $\Delta$-module is called ``small'' if it appears as a subrepresentation of a $\Delta$-module $(F_n)$ generated by $F_1$, and the following abstract result is proved:

\begin{theorem}[Snowden]
\label{thm:delta-old}
Suppose $\bk$ has characteristic $0$ and $F$ is a finitely generated small $\Delta$-module. Then $F$ is noetherian and $\rH_F(\bm)$ is a rational function of the $m_{\lambda}$.
\end{theorem}

Our main result, which we prove below, and ultimately follows from Proposition~\ref{FAhomrep} and Corollary~\ref{FShilb2}, is the following:

\begin{theorem}
\label{thm:delta-new}
Let $F$ be a finitely generated $\Delta$-module. Then $F$ is noetherian and $\rH_F'(\bm)$ is a polynomial in the $m_{\lambda}$ and the $e^{m_{\lambda}}$.
\end{theorem}

This result improves Theorem~\ref{thm:delta-old} in three ways: 
\begin{enumerate}[\indent (1)]
\item there is no restriction on the characteristic of $\bk$; 
\item there is no restriction to small $\Delta$-modules; and 
\item the Hilbert series statement is significantly stronger. 
\end{enumerate}
To elaborate on (3), our result implies that the denominator of $\rH_F(\bm)$ factors into linear pieces of the form $1-\alpha$, where $\alpha$ is an $\bN$-linear combination of the $m_{\lambda}$. This stronger result on Hilbert series answers \cite[Question~5]{delta-mod} affirmatively, and even goes beyond what is asked there. (Note that the question in \cite{delta-mod} is phrased in terms of Schur functions rather than monomial symmetric functions, but they are related by a change of basis with integer coefficients \cite[Corollary 7.10.6]{stanley-EC2}.)

\subsection{Applications}
\label{delta-example}

The motivating example of a $\Delta$-module comes from the study of syzygies of the Segre embedding. Define the {\bf Segre product} of two graded $\bk$-algebras $A$ and $B$ to be the graded $\bk$-algebra $A \boxtimes B$ given by
\begin{displaymath}
(A \boxtimes B)_n = A_n \otimes_{\bk} B_n.
\end{displaymath}
Given finite-dimensional $\bk$-vector spaces $V_1, \ldots, V_n$, put
\begin{displaymath}
R(V_1, \ldots, V_n) = \Sym(V_1) \boxtimes \cdots \boxtimes \Sym(V_n), \qquad
S(V_1, \ldots, V_n) = \Sym(V_1 \otimes \cdots \otimes V_n).
\end{displaymath}
Then $R$ is an $S$-algebra, and applying $\rm Proj$ gives the Segre embedding
\begin{displaymath}
i(V_1, \ldots, V_n) \colon \bP(V_1) \times \cdots \times \bP(V_n) \to \bP(V_1 \otimes \cdots \otimes V_n).
\end{displaymath}
Fix an integer $p \ge 0$, and define
\begin{align} \label{eqn:segre-syz}
F_n(V_1, \ldots, V_n)=\Tor_p^{S(V_1, \ldots, V_n)}(R(V_1, \ldots, V_n), \bk).
\end{align}
This is the space of $p$-syzygies of the Segre. The factorization
\begin{displaymath}
i(V_1, \ldots, V_{n+1})=i(V_1, \ldots, V_{n-1}, V_n \otimes V_{n+1}) \circ (\id \times i(V_n, V_{n+1})),
\end{displaymath}
combined with general properties of $\Tor$, yield transition maps as in \eqref{delta-trans}. In this way, the sequence $F=(F_n)_{n \ge 0}$ naturally has the structure of a $\Delta$-module.

\begin{theorem} \label{thm:segre-delta-fg}
For every $p \ge 0$ and field $\bk$, the $\Delta$-module $F$ defined by \eqref{eqn:segre-syz} is finitely presented and $\rH'_F(\bm)$ is a polynomial in the $m_{\lambda}$ and the $e^{m_{\lambda}}$. 
\end{theorem}

\begin{proof}
The Tor modules are naturally $\bZ_{\ge 0}$-graded and are concentrated between degrees $p$ and $2p$: to see this, we first note that the Segre variety has a quadratic Gr\"obner basis and the dimensions of these Tor modules are bounded from above by those of the initial ideal. The Taylor resolution of a monomial ideal gives the desired bounds \cite[Exercise 17.11]{eisenbud}.

Using the Koszul resolution of $\bk$ as an $S$-module, each graded piece of $F$ can be realized as the homology of a complex of finitely generated $\Delta$-modules; now use Theorem~\ref{thm:delta-new}.
\end{proof}

Finite generation means that there are finitely many $p$-syzygies of Segres that generate all $p$-syzygies of all Segres under the action of general linear groups, symmetric groups, and the maps \eqref{delta-trans}, i.e., pullbacks along Segre embeddings. The statement about Hilbert series means that, with a single polynomial, one can store all of the characters of the $\GL$ actions on spaces of $p$-syzygies. We refer to the introduction of \cite{delta-mod} for a more detailed account.

\begin{remark}
The syzygies of many varieties related to the Segre, such as secant and tangential varieties of the Segre, also admit the structure of a $\Delta$-module. See \cite[\S 4]{delta-mod}. In fact, the argument in Theorem~\ref{thm:segre-delta-fg} almost goes through: one can show that each graded piece of each Tor module is finitely generated and has a rational Hilbert series, but there are few general results about these Tor modules being concentrated in finitely many degrees.

The results that we know of are in characteristic $0$: \cite{raicu} proves that the ideal of the secant variety of the Segre variety is generated by cubic equations, and \cite{oeding-raicu} proves that the tangential variety of the Segre variety is generated by equations of degree $\le 4$.
\end{remark}

\subsection{Proof of Theorem~\ref{thm:delta-new}: noetherianity} \label{ss:delta-noeth}

Let $\Mod_{\Delta, 0}$ be the category whose objects are sequences $F=(F_n)_{n \ge 0}$ where $F_n$ is a polynomial functor $\Vec^n \to \Vec$ (the subscript $0$ is to emphasize that there are no $S_n$-actions and no transition maps). Given an object $F$ of this category, we define a $\Delta$-module $\Gamma(F)$ by
\begin{displaymath}
\Gamma(F)(\{V_i\}_{i \in I}) = \bigoplus_{0 \le n \le |I|} \bigoplus_{\alpha \colon I \surj [n]} F_n(U_1, \ldots, U_n),
\end{displaymath}
where $U_j=\bigotimes_{\alpha(i)=j} V_i$ and the second sum is over all surjective functions $\alpha \colon I \to [n]$.

\begin{proposition} \label{lem:Phi-exact}
\begin{enumerate}[\indent \rm (a)]
\item $\Gamma$ is the left adjoint of the forgetful functor $\Mod_\Delta \to \Mod_{\Delta, 0}$.
\item $\Gamma$ is an exact functor.
\end{enumerate}
\end{proposition}

\begin{proof}
(a) Pick $G \in \Mod_\Delta$ and $F \in \Mod_{\Delta, 0}$ and $f \in \hom_{\Mod_{\Delta, 0}}(F,G)$. Suppose that $\{V_i\}_{i \in I}$ is an object of $\Mod_{\Delta}$ and $\alpha \colon I \to [n]$ is a surjection; let $U_j$ be as above. We have maps
\begin{displaymath}
F_n(U_1, \ldots, U_n) \to G_n(U_1, \ldots, U_n) \to G(\{V_i\}),
\end{displaymath}
where the first is $f$ and the second uses functoriality of $G$ with respect to the morphism $\{U_j\}_{j \in [n]} \to \{V_i\}_{i \in I}$ in $\Vec^{\Delta}$ corresponding to $\alpha$ (and the identity maps). Summing over all choices of $\alpha$ gives a map $(\Gamma F)(\{V_i\}) \to G(\{V_i\})$, and thus a map of $\Delta$-modules $\Gamma(F) \to G$.

(b) A $\Delta$-module is a functor $\Vec^\Delta \to \Vec$, so exactness can be checked pointwise. Similarly, exactness of a sequence of objects in $\Mod_{\Delta, 0}$ can be checked pointwise. So let $0 \to F_1 \to F_2 \to F_3 \to 0$ be an exact sequence in $\Mod_{\Delta, 0}$ and pick $\{V_i\}_{i \in I} \in \Vec^{\Delta}$. The sequence
\[
(\Gamma F_1)(\{V_i\}) \to (\Gamma F_2)(\{V_i\}) \to (\Gamma F_3)(\{V_i\})
\]
is the direct sum (over all surjective functions $\alpha \colon I \surj [n]$) of sequences of the form 
\[
F_1(U_1, \ldots, U_n) \to F_2(U_1, \ldots, U_n) \to F_3(U_1, \ldots, U_n).
\]
The latter are exact by assumption, so our sequence of interest is also exact.
\end{proof}

Define $\cD^{\Lambda}$ to be the $\Delta$-module $\Gamma(\rD^{\Lambda})$, where $\rD^{\Lambda}$ (defined in \S\ref{ss:polyfunc}) is regarded as an object of $\Mod_{\Delta, 0}$ concentrated in degree $r$.

\begin{proposition} \label{prop:delta-proj}
The $\cD^{\Lambda}$ are projective generators for $\Mod_{\Delta}$.
\end{proposition}

\begin{proof}
By Proposition~\ref{prop:Dlambda-projgen}, the functors $\rD^\Lambda$ are projective generators for $\Mod_{\Delta, 0}$. Since $\Gamma$ is the left adjoint of an exact functor, it follows that $\cD^{\Lambda}$ are projective generators. 
\end{proof}

Let $\bd=(d_1, \ldots, d_r)$ be a tuple of positive integers. Let $\rT^{\bd} \colon \Vec^r \to \Vec$ be the functor given by
\begin{displaymath}
\rT^{\bd}(V_1, \ldots, V_r) = V_{1}^{\otimes d_1} \otimes \cdots \otimes V_{r}^{\otimes d_r},
\end{displaymath}
thought of as an object of $\cP^\bN$ by extension by $0$. Define a $\Delta$-module $\cT^{\bd} = \Gamma(\rT^\bd)$. Explicitly, we have
\begin{displaymath}
\cT^{\bd}(\{V_i\}_{i \in I}) = \bigoplus_{\alpha \colon I \surj [r]} \bigotimes_{i \in I} V_i^{\otimes d_{\alpha(i)}},
\end{displaymath}
where the sum is over all surjections $\alpha$ from $I$ to $[r]$.

\begin{corollary} \label{Trsubq}
\mbox{A finitely generated $\Delta$-module is a subquotient of a finite direct sum of $\cT^{\bd}$.}
\end{corollary}

\begin{proof}
Since $\rD^{\Lambda}$ is a subobject of $\rT^{\bd}$, where $d_i=\vert \lambda^i \vert$, it follows that $\cD^{\Lambda} = \Gamma(\rD^\Lambda)$ is a subrepresentation of $\cT^{\bd} = \Gamma(\rT^\bd)$. Now use Proposition~\ref{prop:delta-proj}.
\end{proof}

We now proceed to show that $\cT^{\bd}$ is a noetherian $\Delta$-module and that every finitely generated $\Delta$-module is a subquotient of a finite direct sum of $\cT^{\bd}$'s. This will show that finitely generated $\Delta$-modules are noetherian.

Let $N$ be a positive integer and define a functor $\Phi_N \colon \FS^{\op} \to \Vec^{\Delta}$ as follows. Given a finite set $S$, we let $\Phi_N(S)$ be the family $\{V_i\}_{i \in S}$ where $V_i=\bk^N$ for all $i$. Given a surjection $f \colon S \to T$, we let $\Phi_N(f)$ be the map $\Phi_N(T) \to \Phi_N(S)$ that is $f$ on index sets, and where the linear map $\eta_t \colon \bk^N \to \bigotimes_{f(s)=t} \bk^N$ takes the basis vector $e_i$ to $\bigotimes_{f(s)=t} e_i$. We thus obtain a functor $\Phi_N^* \colon \Mod_{\Delta} \to \Rep_{\bk}(\FS^{\op})$.

\begin{lemma} \label{lem:Tbd-FS-noeth}
For all $N$, the representation $\Phi_N^*(\cT^{\bd})$ of $\FS^\op$ is noetherian.
\end{lemma}

\begin{proof}
Put $M=\Phi^*_N(\cT^{\bd})$. We have
\begin{displaymath}
M(S) = \bigoplus_{\alpha \colon S \surj [r]} \bigotimes_{i \in S} (\bk^N)^{\otimes d_{\alpha(i)}}.
\end{displaymath}
Let $\cF(d)$ be the set of all functions $[d] \to [N]$, which naturally indexes a basis of $(\bk^N)^{\otimes d}$. The space $M(S)$ has a basis indexed by pairs $(\alpha,\beta)$, where $\alpha$ is a surjection $S \to [r]$ and $\beta$ is a function assigning to each element $i \in S$ an element of $\cF(d_{\alpha(i)})$. Write $e_{(\alpha, \beta)}$ for the basis vector corresponding to $(\alpha, \beta)$. If $f \colon T \to S$ is a surjection and $f^* \colon M(S) \to M(T)$ is the induced map then $f^*(e_{(\alpha,\beta)})=e_{(f^*(\alpha), f^*(\beta))}$. A pair $(\alpha,\beta)$ defines a function $S \to [r] \times \cF$, where $\cF = \cF(d_1) \amalg \cdots \amalg \cF(d_r)$. The previous remark shows that $M$ is a subrepresentation of the $\FS^\op$-module defined by $T \mapsto \bk[\Hom_{\FA}(T, [r] \times \cF)]$, which is noetherian (Proposition~\ref{FAhomrep}), so the result follows.
\end{proof}

\begin{proposition} \label{Tdnoeth}
The $\Delta$-module $\cT^{\bd}$ is noetherian.
\end{proposition}

\begin{proof}
Let $M$ be a subrepresentation of $\cT^{\bd}$. We claim that $M$ is generated by its restriction to the subcategory $\Phi_N(\FS^\op)$, where $N=\max(d_i)$. To see this, pick an element $v \in M_n(\bk^{s_1}, \ldots, \bk^{s_n})$ for some $s_i$. Then $v$ is a linear combination of weight vectors, so we may assume that $v$ is a weight vector. A weight vector in $\cT^{\bd}(\bk^{s_1}, \ldots, \bk^{s_n})$ uses at most $N$ basis vectors from each space. Thus there is an element $g \in \prod_{i=1}^n \GL(s_i)$ such that $gv \in \cT^{\bd}(\bk^N, \ldots, \bk^N)$. This proves the claim. It follows that if $M \subsetneqq M' \subseteq \cT^{\bd}$ then $\Phi^*_N(M) \subsetneqq \Phi^*_N(M')$. Now use Lemma~\ref{lem:Tbd-FS-noeth}.
\end{proof}

Corollary~\ref{Trsubq} and Proposition~\ref{Tdnoeth} prove the noetherianity statement in Theorem~\ref{thm:delta-new}.

\subsection{Proof of Theorem~\ref{thm:delta-new}: Hilbert series} \label{sec:proof-delta-new}

By Corollary~\ref{Trsubq}, the Grothendieck group of finitely generated $\Delta$-modules is spanned by the classes of subrepresentations of $\cT^{\bd}$. Therefore, it suffices to analyze the Hilbert series of a subrepresentation $M$ of $\cT^{\bd}$.

Let $N=\max(d_i)$, and write $\Phi = \Phi_N$ (\S\ref{ss:delta-noeth}). The space $\Phi^*(M)(S)$ has an action of the group $\GL(N)^S$ that we exploit. Let $\lambda_1, \ldots, \lambda_n$ be the partitions of size at most $N$. These are the only partitions that appear as weights in $(\bk^N)^{\otimes d_i}$. Let $(S_1, \ldots, S_n)$ be a tuple of finite (possibly empty) sets and let $S$ be their disjoint union. Define $\Psi(M)(S_1, \ldots, S_n)$ to be the subspace of $\Phi^*(M)(S)$ where the torus in the $i$th copy of $\GL(N)$ acts by weight $\lambda_j$, where $i \in S_j$. If $f_i \colon S_i \to T_i$ are surjections, and $f \colon S \to T$ is their disjoint union, then the map $f^* \colon \Phi^*(M)(T) \to \Phi^*(M)(S)$ carries $\Psi(M)(T_1, \ldots, T_n)$ into $\Psi(M)(S_1, \ldots, S_n)$. Thus $\Psi(M)$ is an $(\FS^{\op})^n$-module. If $M'$ is a subrepresentation of $\Psi(M)$ then 
\[
S \mapsto \bigoplus_{S=S_1 \amalg \cdots \amalg S_n} M'(S_1, \ldots, S_n)
\]
is an $\FS^{\op}$-subrepresentation of $\Phi^*(M)$. Since $\Phi^*(M)$ is noetherian, it follows that $\Psi(M)$ is noetherian. In particular, $\Psi(M)$ is finitely generated.

Given an integer $e \ge 0$, set $V=M_e(\bk^N, \ldots, \bk^N)$. The weight space $V_{\mu_1,\ldots,\mu_e}$ is isomorphic to $\Psi(M)([e_1], \ldots, [e_n])$ where $e_i = \#\{j \mid \mu_j = \lambda_i\}$. The coefficient of $m_{\lambda_1}^{e_1} \cdots m_{\lambda_n}^{e_n}$ in $\rH'_M(\bm)$ is then
\begin{displaymath}
\frac{1}{e!} \sum_{\substack{\mu_1, \ldots, \mu_e\\ \#\{j \mid \mu_j = \lambda_i\} = e_i}} \dim_{\bk}(V_{\mu_1,\ldots,\mu_e}) = \frac{1}{e_1! \cdots e_n!} \dim_{\bk}(\Psi(M)([e_1], \ldots, [e_n])).
\end{displaymath}
It follows that $\rH'_M(\bm)$ is equal to $\rH'_{\Psi(M)}(\bt)$, and so the theorem follows from Corollary~\ref{FShilb2}.

\section{Additional examples} \label{sec:additional}

\subsection{Categories of $G$-sets} \label{ss:G-sets}

The proofs of the results in this section appear in \cite{catgb-other}.

Let $G$ be a group. A {\bf $G$-map} $S \to T$ between finite sets $S$ and $T$ is a pair $(f, \sigma)$ consisting functions $f \colon S \to T$ and $\sigma \colon S \to G$. Given $G$-maps $(f, \sigma) \colon S \to T$ and $(g, \tau) \colon T \to U$, their composition is the $G$-map $(h, \eta) \colon S \to U$ with $h=g \circ f$ and $\eta(x)=\sigma(x) \tau(f(x))$, the product taken in $G$. In this way, we have a category $\FA_G$ whose objects are finite sets and whose morphisms are $G$-maps. The automorphism group of $[n]$ in $\FA_G$ is $S_n \wr G = S_n \ltimes G^n$. Let $\FS_G$ (resp.\ $\FI_G$) denote the subcategory of $\FA_G$ containing all objects but only those morphisms $(f,\sigma)$ with $f$ surjective (resp.\ injective).

\begin{remark}
The category $\FI_{\bZ/2\bZ}$ is equivalent to the category $\FI_{\rm BC}$ defined in \cite[Defn.~1.2]{wilson}. Representations of $\prod_{n=0}^{\infty} \FS^{\op}_{S_n}$ are equivalent to $\Delta$-modules (in good characteristic). Thus the representation theory $\FI_G$ and $\FS_G^{\op}$ generalizes many examples of interest. This motivates the results stated below.
\end{remark}

Recall that a group $G$ is {\bf polycyclic} if it has a finite composition series $1 = G_0 \subseteq G_1 \subseteq \cdots \subseteq G_r = G$ such that $G_i / G_{i-1}$ is cyclic for $i=1,\dots,r$, and it is {\bf polycyclic-by-finite} if it contains a polycyclic subgroup of finite index. 

\begin{theorem}
If $G$ is finite then $\FI_G$, $\FS^\op_G$, $\FA_G$, and $\FA_G^\op$ are quasi-Gr\"obner. Finitely generated modules over any of these categories have rational Hilbert series.

If $G$ is polycyclic-by-finite and $\bk$ is left-noetherian, then $\Rep_\bk(\FI_G)$ and $\Rep_\bk(\FA_G)$ are noetherian.
\end{theorem}

Let $M$ be a finitely generated $\FS_G^{\op}$-module over an algebraically closed field $\bk$. Then $M([n])$ is a finite-dimensional representation of $G^{n}$. Let $[M]_{n}$ be the image of the class of this representation under the map $\cR_{\bk}(G^n)= \cR_{\bk}(G)^{\otimes n} \to \Sym^{n}(\cR_{\bk}(G))$ ($\cR_\bk(H)$ is the Grothendieck group of all finitely generated $\bk[H]$-modules). One can recover the isomorphism class of $M([n])$ as a representation of $G^{n}$ from $[M]_{n}$ due to the $S_{n}$-equivariance. If $\{L_{j}\}$ are the irreducible representations of the $G$, then $[M]_{n}$ is a polynomial in corresponding variables $t_j = [L_j]$. Define the {\bf enhanced Hilbert series} of $M$ (see \S\ref{enhanced}) by 
\begin{displaymath}
\bH_M(t) = \sum_{n \ge 0} [M]_{n} \in \SYM(\cR_{\bk}(G)_{\bQ}) \cong \bQ \lbb t_{j} \rbb.
\end{displaymath}
Let $N$ be the exponent of $G$ (LCM of all orders of elements) and fix a primitive $N$th root of unity $\zeta_N$. The following is a simplified version of our main theorem on Hilbert series. 

\begin{theorem} \label{FSGhilb}
Notation as above. Then $\bH_M(\bt) = f(\bt) / g(\bt)$ where $f(\bt) \in \bQ(\zeta_N)[\bt]$ and $g(\bt) = \prod_k (1-\lambda_k)$ where each $\lambda_k$ is a $\bZ[\zeta_N]$-linear combination of the $t_i$.
\end{theorem}

\subsection{An example with non-regular languages} \label{sec:nonregular-lang}

Let $\OI_d^=$ be the subcategory of $\OI_d$ containing all objects but only those morphisms $(f,g)$ for which all fibers of $g$ have the same size. The inclusion $\OI^=_d \subset \OI_d$ satisfies property~(S) (Definition~\ref{propS}), so $\OI_d^=$ is Gr\"obner by Proposition~\ref{grobsub} and Theorem~\ref{oithm}. Endow $\OI_d^=$ with the restricted norm from $\OI_d$. Obviously, $\OI_1^==\OI$, which is O-lingual (Theorem~\ref{oithm}). We now examine $d>1$.

\begin{proposition}
The normed category $\OI^=_2$ is $\mathrm{UCF}$-lingual.
\end{proposition}

\begin{proof}
Let $\cC=\OI_2^=$ and let $\cC'=\OI_2$. Let $x=[n]$ be an object of $\cC$ and let $\Sigma=\{0,1,2\}$. We regard $\vert \cC_x \vert$ as a subset of $\vert \cC_x' \vert$ with the induced order, which is admissible by Proposition~\ref{grobsub}. Define the map $i \colon \vert \cC'_x \vert \to \Sigma^{\star}$ as in the proof of Theorem~\ref{oithm}. 

If $T$ is an ideal of $\vert \cC_x \vert$ and $S$ is the ideal of $\vert \cC_x' \vert$ it generates, then $T = S \cap \vert \cC_x \vert$. Thus $i(T)$ is the intersection of the regular language $i(S)$ with the language $\cL=i(\vert \cC_x \vert)$. The language $\cL$ consists of words in $\Sigma$ that contain exactly $n$ 0's and use the symbols $1$ and $2$ the same number of times. This is a deterministic context-free (DCF) language (the proof is similar to \cite[Exercise~5.1]{hopcroftullman}), so the same is true for $i(T)$ \cite[Thm.~10.4]{hopcroftullman}. Finally, DCF implies UCF (see the proof of \cite[Thm.~5.4]{hopcroftullman}).
\end{proof}

\begin{corollary}
If $M$ is a finitely-generated representation of $\OI^=_2$ then $\rH_M(t)$ is an algebraic function of $t$.
\end{corollary}

\begin{example}
Let $d \ge 1$ be arbitrary. Let $M_d \in \Rep_{\bk}(\OI_d^=)$ be the principal projective at $[0]$, and let $\rH_d$ be its Hilbert series. The space $M_d([n])$ has for a basis the set of all strings in $\{1,\ldots,d\}$ of length $n$ in which the numbers $1, \ldots, d$ occur equally. The number of strings in which $i$ occurs exactly $n_i$ times is the multinomial coefficient $\binom{n_1+\cdots+n_d}{n_1, \dots, n_d}$.

It follows that $\rH_d(t) = \sum_{n \ge 0} \frac{(dn)!}{n!^d} t^{dn}$.
So $\rH_1(t) = (1-t)^{-1}$ and $\rH_2(t) = (1-4t^2)^{-1/2}$ but $\rH_d(t)$ is not algebraic for $d>2$ \cite[Thm.~3.8]{woodcocksharif}.
\end{example}

\begin{corollary}
If $d>2$ then $\OI_d^=$ is not UCF-lingual.
\end{corollary}

Let $\FI_d^=$ be the subcategory of $\FI_d$ defined in a way similar to $\OI_d^=$ with the induced norm. Then $\FI_d^=$ is quasi-Gr\"obner. The above results and examples show that the Hilbert series of a finitely generated $\FI_2^=$-module is algebraic, but that this need not be the case for $\FI_d^=$-modules with $d>2$. We have shown, by very different means, that for $d>2$ the Hilbert series are always D-finite (in characteristic 0).

The tca $\Sym(\Sym^2(\bC^{\infty}))$ is the coordinate ring of infinite size symmetric matrices, so has a natural rank stratification, and a module is {\bf bounded} if it lives in one of the finite strata. The above results on $\FI_2^=$ imply that a module supported in rank $\le 2$ have algebraic Hilbert series. This motivates the following conjecture, which we have verified for rank $\le 3$:

\begin{conjecture}
A bounded $\Sym(\Sym^2(\bC^\infty))$-module has an algebraic Hilbert series.
\end{conjecture}

\addtocontents{toc}{\vskip.3\baselineskip}

\section{Open problems} \label{sec:problems}

\subsection{Krull dimension}
There is a notion of Krull dimension for abelian categories \cite[Ch.~IV]{gabriel} generalizing that for commutative rings. One would like a combinatorial method to compute the Krull dimension of $\Rep_{\bk}(\cC)$. In Proposition~\ref{dim0} we give a criterion for dimension 0, but it is probably far from optimal. 

If $\bk$ is a field of characteristic~$0$, then $\Rep_{\bk}(\FA)$ has Krull dimension~$0$ \cite{wiltshire}, $\Rep_{\bk}(\FI)$ has Krull dimension~$1$ \cite[Corollary 2.2.6]{symc1}, and $\Rep_{\bk}(\OI)$ has infinite Krull dimension (easy). These results hold in positive characteristic as well: $\Rep_\bk(\FA)$ is handled in Theorem~\ref{thm:FA-artinian}, $\Rep_\bk(\OI)$ remains easy, and $\Rep_\bk(\FI)$ is an unpublished result of ours.

\subsection{Enhanced Hilbert series} \label{enhanced}
Our definition of the Hilbert series of $M \in \Rep_{\bk}(\cC)$ only records the dimension of $M(x)$, for each object $x$. One could attempt to improve this by recording the representation of $\Aut(x)$ on $M(x)$. Suppose $R$ is a ring and for each $x \in \vert \cC \vert$ we have an additive function $\mu_x \colon \cR_{\bk}(\Aut(x)) \to R$ ($\cR_\bk$ denotes the Grothendieck group). We define the {\bf enhanced Hilbert series} of $M$ by
\begin{displaymath}
\wt{\rH}_M = \sum_{x \in \vert \cC \vert} \mu_x([M(x)]),
\end{displaymath}
where $[M(x)]$ is the class of $M(x)$ in $\cR_{\bk}(\Aut(x))$, when this sum makes sense. For example, if $\cC=\FI$ and $\bk$ has characteristic $0$, we define maps $\mu_x$ to $R=\bQ \lbb t_1,t_2,\ldots \rbb$ in \cite[\S 5.1]{symc1}, and prove a sort of rationality result there. We can now prove the analogous result for $\FI_d$-modules as well. A rationality result for  enhanced Hilbert series for $\FS_G^{\op}$-modules is given in \S\ref{ss:G-sets}. Is it possible to prove a general rationality result?

\subsection{Minimal resolutions and Poincar\'e series}
Suppose $\cC$ is a weakly directed (i.e., any self-map is invertible, and also known as an ``EI-category'') normed category. For a $\cC$-module $M$, let $\Psi(M)$ be the $\cC$-module defined as follows: $\Psi(M)(x)$ is the quotient of $M(x)$ by the images of all maps $M(y) \to M(x)$ induced by non-isomorphisms $y \to x$ in $\cC$. One can think of $\Psi(M)$ as analogous to tensoring $M$ with the residue field in the case of modules over an augmented algebra. The left-derived functors of $\Psi$ exist, and if the group algebras $\bk[\Aut(x)]$ are semisimple, one can read off from $\rL^i \Psi(M)$ the $i$th projective in the minimal resolution of $M$. Define the {\bf Poincar\'e series} of $M$ by
\begin{displaymath}
\rP_M(\bt, q) = \sum_{i \ge 0} \rH_{\rL^i \Psi(M)}(\bt) (-q)^i.
\end{displaymath}
This contains strictly more information than the Hilbert series, but is much more subtle since it does not factor through the Grothendieck group. What can one say about the form of this series? We proved a rationality theorem for Poincar\'e series of $\FI$-modules in \cite[\S 6.7]{symc1} and can generalize this result to $\FI_d$-modules. Preliminary computations with $\VI_{\bF_q}$-modules (defined in \S\ref{sec:linear-cat}) have seen theta series come into play; it is not clear yet if there is a deeper meaning to this.

\subsection{Noetherianity results for other categories} \label{open-noeth}
There are combinatorial categories which are expected to be noetherian, but do not fall into our framework. For example:
Let $\cC$ be the category whose objects are finite sets, and where a morphism $S \to T$ is an injection $f \colon S \to T$ together with a perfect matching on $T \setminus f(S)$. When $\bk$ has characteristic $0$, $\Rep_{\bk}(\cC)$ is equivalent to the category of $\Sym(\Sym^2(\bk^{\infty}))$-modules with a compatible polynomial action of $\GL_{\infty}(\bk)$. See also \cite[\S 4.2]{infrank} for a connection between $\Rep_\bk(\cC)$ (where $\cC$ is called (db)) and the stable representation theory of the orthogonal group.

When $\bk$ is a field of characteristic $0$, the main result of \cite{sym2noeth} shows that $\Rep_\bk(\cC)$ is noetherian. We expect $\Rep_{\bk}(\cC)$ is noetherian for general noetherian $\bk$, but cannot prove it. There is an ordered version $\cD$ where the objects are ordered finite sets and the maps $f \colon S \to T$ are order-preserving, but the principal projective $P_\emptyset$ is not noetherian: if $M_n \colon \emptyset \to [2n]$ is the perfect matching on $\{1, \dots, 2n\}$ consisting of the edges $(i,i+3)$ where $i=1,3,\dots,2n-3$ and the edge $(2n-1,2)$, then $M_3, M_4, M_5, \dots$ are pairwise incomparable. Is there a way to extend the scope of the methods of this paper to include these categories?

\subsection{Coherence}
The category $\Rep_{\bk}(\cC)$ is {\bf coherent} if the kernel of any map of finitely generated projective representations is finitely generated. This is a weaker property than noetherianity, and should therefore be easier to prove. We have some partial combinatorial results on coherence, but none that apply in cases of interest. We would be especially interested in a criterion that applies to the category mentioned in \S\ref{open-noeth}.

\subsection{Hilbert series of ideals in permutation posets} \label{q:permutations}
There are many algebraic structures not mentioned in this paper which lead to interesting combinatorial problems. As an example, let $\fS$  be the disjoint union of all finite symmetric groups. Represent a permutation $\sigma$ in one-line notation: $\sigma(1) \sigma(2) \cdots \sigma(n)$. Say that $\tau \le \sigma$ if there is a consecutive subword $\sigma(i) \sigma(i+1) \cdots \sigma(i+r-1) \sigma(i+r)$ which gives the same permutation as $\tau$, i.e., $\sigma(j) > \sigma(j')$ if and only if $\tau(j) > \tau(j')$ for all $j \ne j'$. (In the literature, $\tau$ is a {\bf consecutive pattern} in $\sigma$.) If we drop the condition ``consecutive,'' then this is the poset of pattern containment (see \cite[\S 7.2.3]{bona}) which has infinite anti-chains (\cite[Theorem 7.35]{bona}), so the same is true for consecutive pattern containment. 

What is the behavior of Hilbert series of finitely generated ideals in this poset? By \cite{garrabrant-pak}, there are examples for the pattern containment poset which are not D-finite. The relevant algebra is monomial ideals in the free shuffle algebra (see \cite[Example 2.2(b)]{ronco} or \cite[\S 2.2, Example 2]{DK}).

\addtocontents{toc}{\vskip.5\baselineskip}

\end{document}